\documentclass[reqno,12pt]{amsart}
\usepackage{amssymb}

\newtheorem{lemma}{Lemma}[section]
\newtheorem{theorem}[lemma]{Theorem}
\newtheorem{corollary}[lemma]{Corollary}
\newtheorem{remark}[lemma]{Remark}
\newtheorem{proposition}[lemma]{Proposition}
\newtheorem{definition}[lemma]{Definition}

\newcommand{\mN}{\mathbb{N}}

\newcommand{\mC}{\mathbb{C}}
\newcommand{\mR}{\mathbb{R}}

\def\Hom{\mathrm{Hom}}
\def\End{\mathrm{End}}
\def\grad{\mathrm{grad}}
\def\scal{\mathrm{scal}}
\def\im{\mathrm{im}\,}
\def\Re{\mathrm{Re}}
\def\SO{\mathrm{SO}}
\def\Spin{\mathrm{Spin}}
\def\rP{\mathrm{P}}

\def\dv{\mathrm{dv}}
\def\id{\mathrm{Id}}
\def\spann{\mathrm{span}}

\def\supp{\mathrm{supp}}

\def\geucl{g_{\mathrm{eucl}}}

\begin{document}
\title{Zero sets of eigenspinors for generic metrics}
\author{Andreas Hermann}
\date{}
\address{Institut f\"ur Mathematik \\
   Universit\"at Potsdam\\
   Am Neuen Palais 10\\
   14469 Potsdam \\
   Germany} 
\email{hermanna@uni-potsdam.de}

\begin{abstract}
Let $M$ be a closed connected spin manifold of dimension 
$2$ or $3$ with a fixed orientation and a fixed spin structure. 
We prove that for a generic Riemannian metric on $M$ the non-harmonic 
eigenspinors of the Dirac operator are nowhere zero. 
The proof is based on a transversality theorem and the unique continuation 
property of the Dirac operator. 
\end{abstract}

\maketitle

\tableofcontents

\section{Introduction}

The spectrum of the Dirac operator can be 
computed explicitly for some closed Riemannian 
spin manifolds. 
In these examples one very often has eigenvalues 
of high multiplicities 
(see e.\,g.\,\cite{baer:96}, \cite{friedrich:84}). 
However it is also known that for a generic 
choice of metric on a closed spin manifold 
of dimension~$2$ or~$3$ 
the eigenvalues are simple (see \cite{dahl:03}). 
In the present article we extend this result by 
showing that for a generic metric the non-harmonic 
eigenspinors are nowhere zero. 
More precisely let~$M$ be a closed spin manifold 
and denote by~$R(M)$ the space of all 
smooth Riemannian metrics on~$M$
equipped with the $C^1$-topology.  
For every~$g\in R(M)$ denote by~$D^g$ the Dirac 
operator acting on spinors for the metric~$g$. 
Let~$N(M)$ be the subset of all~$g\in R(M)$ 
such that all the eigenspinors 
of~$D^g$ are nowhere zero on~$M$ 
and let~$N^*(M)$ be the subset of all 
$g\in R(M)$ such that all the non-harmonic 
eigenspinors of~$D^g$ are nowhere zero on~$M$. 
Then we prove the following. 

\begin{theorem}
\label{main_theorem}
Let $M$ be a closed connected spin manifold 
with a fixed orientation and a fixed spin 
structure. 
If $\dim M=2$ then~$N^*(M)$ is residual in~$R(M)$. 
If $\dim M=3$ then~$N(M)$ is residual in~$R(M)$. 
\end{theorem}

Recall that a subset is residual, if it contains 
a countable intersection of open and dense sets. 
The proof will show that for every~$g\in R(M)$ 
the intersection of~$N^*(M)$ 
with the conformal class~$[g]$ of~$g$ is 
residual in~$[g]$ 
(see Theorem \ref{theorem_no_zeros}). 
In the case~$\dim M=3$ we use the fact that 
the subset of all~$g\in R(M)$ 
with~$\ker(D^g)=\{0\}$ is open and dense in~$R(M)$ 
(see \cite{maier:97}, \cite{ammann.dahl.humbert:11}). 
Furthermore for~$\dim M=2$ the subset~$N(M)$ is 
in general not residual in~$R(M)$. 
Namely in Section~\ref{surfaces_section} 
we will give examples 
of closed surfaces~$M$ such that for 
every metric~$g$ on~$M$ there exist harmonic 
spinors of~$D^g$ with non-empty zero set. 
The proof will also show that 
Theorem~\ref{main_theorem} is true with 
respect to every $C^k$-topology, $k\geq1$, 
on~$R(M)$. 
In order to simplify the notation we will 
state all results using the $C^1$-topology.

In analysis, geometry and mathematical physics 
the study of zero sets of solutions to 
generalized Dirac equations is very important.   
On closed manifolds these zero sets have 
codimension~$2$ at least by a result of B\"ar 
(see \cite{baer:97}). 
There are many open questions in the literature 
concerning these zero sets. 
For example Nester and his co-authors 
ask whether on a generic 
asymptotically Euclidean manifold of dimension~$3$ 
one can find a Witten spinor without zeros 
(see e.\,g.\,\cite{frauendiener.nester.szabados:11}). 
As another example we mention the question 
raised by Ammann, whether for a generic Riemannian 
metric on a closed spin manifold 
of dimension~$n\geq2$ one can find 
a nowhere vanishing spinor~$\psi$ satisfying 
a non-linear equation of the form 
\begin{displaymath}
D^g\psi=C|\psi|^{2/(n-1)}\psi,\quad C>0
\textrm{ constant},
\end{displaymath}
(see \cite{ammann:09}, \cite{ammann.moroianu.moroianu:13}). 
We hope that some of the techniques developed in 
this article can be applied to these questions. 
However we note that in both cases one has 
to consider non-conformal deformations of the 
metric, while in this article we consider mainly 
conformal deformations.

This article is organized as follows. 
In Section~\ref{identification_section} we 
give a short review of spin geometry in order 
to fix the notation and we recall a method 
to compare spinors for different metrics. 
In Section~\ref{further_prelim_section} 
we state a transversality theorem from differential 
topology and Aronszajn's unique 
continuation theorem. 
These are the main tools for the proof of 
Theorem~\ref{main_theorem}. 
We also need the expansion of Green's function 
for the operator~$D^g-\lambda$ at the singularity 
with~$\lambda\in\mR$. 
This expansion is derived in 
Section~\ref{green_chapter}. 
Then in Section~\ref{gen_eigenspinor_section}
we prove Theorem \ref{main_theorem}. 
We first construct a continuous 
map~$F$ which assigns to 
every Riemannian metric~$h$ in an open neighborhood 
of a fixed metric~$g$ an eigenspinor 
of~$D^h$ viewed as a section of~$\Sigma^gM$. 
Theorem \ref{main_theorem} follows from 
the transversality theorem if we can prove that 
the evaluation map corresponding to~$F$ 
is transverse to the zero section of~$\Sigma^gM$. 
Assuming that this is not the case 
we obtain an equation involving Green's function 
for the operator~$D^g-\lambda$ with $\lambda\in\mR$. 
Using the expansion of Green's function 
and using that~$\lambda\neq0$ and 
that~$\dim M\in\{2,3\}$ 
we obtain a contradiction from the 
unique continuation theorem. 
Finally, in Section \ref{surfaces_section} 
we give an example showing that Theorem 
\ref{main_theorem} 
does not hold for harmonic spinors on 
closed surfaces.

\textbf{Acknowledgments}
The author thanks Bernd Ammann and Mattias 
Dahl for many interesting discussions.

\section{Preliminaries}

\subsection{Review of spin geometry}
\label{identification_section}

Let $(M,g)$ be an oriented Riemannian spin manifold 
of dimension~$n$. We denote 
by~$\rP_{\SO}(M,g)$ the principal~$\SO(n)$-bundle of 
positively oriented 
$g$-orthonormal frames. 
A spin structure on~$(M,g)$ consists of a principal 
$\Spin(n)$-bundle 
$\rP_{\Spin}(M,g)$ and a two-fold covering 
\begin{displaymath}
\Theta:\quad\rP_{\Spin}(M,g)\to\rP_{\SO}(M,g),
\end{displaymath} 
which is compatible with the group actions 
of~$\SO(n)$ on~$\rP_{\SO}(M,g)$ 
and of~$\Spin(n)$ on~$\rP_{\Spin}(M,g)$. 
We will always assume that a spin manifold has a 
fixed orientation and a fixed 
spin structure. 
If in addition a Riemannian metric~$g$ on~$M$ is 
chosen, we denote the 
Riemannian spin manifold by~$(M,g,\Theta)$. 
The spinor bundle is defined as the associated 
vector bundle 
$\Sigma^gM:=\rP_{\Spin}(M,g)\times_{\rho}\Sigma_n$, 
where~$\rho$ is 
the spinor representation on the complex vector 
space~$\Sigma_n$ 
of dimension~$2^{[n/2]}$. 
We will denote by~$\langle.,.\rangle$ 
the usual Hermitian inner product on~$\Sigma^gM$. 
The Levi-Civita connection on~$(M,g)$ induces 
a connection on~$\Sigma^gM$ 
denoted by~$\nabla^g$. 
For every~$p\in M$ the Clifford multiplication by 
tangent 
vectors in~$p$ will be denoted by 
\begin{displaymath}
T_pM\otimes\Sigma^g_pM\to\Sigma^g_pM,
\quad X\otimes\psi\mapsto X\cdot\psi.
\end{displaymath}
The Dirac operator~$D^g$ is a linear elliptic 
differential 
operator of first order acting on smooth sections 
of~$\Sigma^gM$: 
\begin{displaymath}
  D^g:\quad C^{\infty}(\Sigma^gM)
  \to C^{\infty}(\Sigma^gM).
\end{displaymath}
If $(e_i)_{i=1}^{n}$ is a local $g$-orthonormal 
frame of~$TM$,
then the Dirac operator is locally given by 
\begin{displaymath}
D^g\psi=\sum_{i=1}^{n}e_i\cdot\nabla^g_{e_i}\psi. 
\end{displaymath}
Furthermore for all 
$\psi\in C^{\infty}(\Sigma^gM)$ and for all 
$f\in C^{\infty}(M,\mC)$ we have 
\begin{displaymath}
D^g(f\psi)=\grad^g(f)\cdot\psi+fD^g\psi
\end{displaymath}
(see \cite{lawson.michelsohn:89}, p.\,116). 
For a detailed introduction to the concepts of 
spin geometry which we use here 
see \cite{lawson.michelsohn:89} or \cite{friedrich:00}. 

Let $(M,\Theta)$ be a closed spin manifold 
and let~$g,h\in R(M)$. 
Then the spinor bundles~$\Sigma^gM$ and~$\Sigma^hM$ 
are two different vector bundles. 
The problem of identifying spinors 
and the Dirac operators~$D^g$ and~$D^h$ 
for the metrics~$g$ and~$h$ has been treated 
in the literature (see \cite{hitchin:74}, \cite{hijazi:86}, 
\cite{bourguignon.gauduchon:92}, \cite{maier:97}) with the following result: 
There exists an isomorphism of vector bundles 
\begin{displaymath}
\beta_{g,h}:\quad 
\Sigma^gM\to\Sigma^hM
\end{displaymath}
which is a fiberwise isometry with respect to 
the inner products on~$\Sigma^gM$ and~$\Sigma^hM$. 
There also exists an isomorphism of vector bundles 
\begin{displaymath}
\overline{\beta}_{g,h}:\quad 
\Sigma^gM\to\Sigma^hM
\end{displaymath}
which induces an isometry of Hilbert spaces 
$L^2(\Sigma^gM)\to L^2(\Sigma^hM)$. 
It is obtained from~$\beta_{g,h}$ by pointwise 
multiplication with a positive function which takes 
into account the change of the volume form. 
We have~$\overline{\beta}_{h,g}
=\overline{\beta}_{g,h}^{-1}$. 
By this isomorphism the Dirac operator~$D^h$ 
can be regarded as a differential operator 
acting on spinors for the metric~$g$. 
More precisely one defines 
\begin{displaymath}
D^{g,h}:=\overline{\beta}_{h,g} D^h 
\overline{\beta}_{g,h}:\quad 
C^{\infty}(\Sigma^gM)\to
C^{\infty}(\Sigma^gM).
\end{displaymath}
This operator has a self-adjoint closure 
in~$L^2(\Sigma^gM)$. 

For all $g,h\in R(M)$ 
there exists an open interval~$I$ 
containing $[0,1]$ such that for every~$t\in I$ 
the tensor field~$g_t:=g+t(h-g)$ is a 
Riemannian metric on~$M$. 
This family of Riemannian metrics on~$M$ 
induces real analytic families of eigenvalues and 
eigenspinors of 
the family of operators~$(D^{g,g_t})_{t\in I}$ 
(see \cite{bourguignon.gauduchon:92}). 
Namely by a theorem of Rellich (see Thm. VII.3.9 in 
\cite{kato:95}) 
we have the following lemma. 

\begin{lemma}
\label{rellich_lemma}
Let $\lambda$ be an eigenvalue of $D^g$ with 
$d:=\dim_{\mC}\ker(D^g-\lambda)$. 
Then there exist real analytic functions 
\begin{displaymath}
I\ni t\mapsto\lambda_{j,t}\in\mR,\quad 
1\leq j\leq d,
\end{displaymath}
such that for every~$j$ and every~$t$ the 
number~$\lambda_{j,t}$ is an eigenvalue 
of~$D^{g,g_t}$ and such 
that~$\lambda_{j,0}=\lambda$ for every~$j$. 
Furthermore there are 
spinors~$\psi_{j,t}\in C^{\infty}(\Sigma^gM)$, 
$1\leq j\leq d$, $t\in I$, 
which are real analytic in~$t$, such that for 
every~$t\in I$ 
the spinors~$\psi_{j,t}$, $1\leq j\leq d$, form 
an $L^2$-orthonormal system 
and for every~$t\in I$ and for every~$j$ the 
spinor~$\psi_{j,t}$ 
is an eigenspinor of~$D^{g,g_t}$ corresponding 
to~$\lambda_{j,t}$. 
\end{lemma}

For all $j\in\{1,...,d\}$ and for all $t\in I$ 
we have $\lambda_{j,t}
=(\psi_{j,t},D^{g,g_t}\psi_{j,t})_{L^2}$. 
By taking the derivative at~$t=0$ and using the fact 
that~$D^g$ is self-adjoint and that for every~$t$ 
we have~$\|\psi_{j,t}\|_{L^2}=1$ we obtain 
\begin{equation}
\label{lambda_deriv_general}
\frac{d\lambda_{j,t}}{dt}\big|_{t=0}
=\int_M\langle\psi_{j,0},
\frac{d}{dt}D^{g,g_t}\big|_{t=0}
\psi_{j,0}\rangle\,\dv^g,\quad 1\leq j\leq d.
\end{equation}

If $g$ and $h$ are conformally related, 
i.\,e.\,$h=e^{2u}g$ 
with~$u\in C^{\infty}(M,\mR)$, 
then for all~$\psi\in C^{\infty}(\Sigma^gM)$ 
we have 
\begin{displaymath}
  D^h(e^{-(n-1)u/2}\beta_{g,h}\psi)=e^{-(n+1)u/2}\beta_{g,h}D^g\psi
\end{displaymath}
(see e.\,g.\,\cite{hijazi:86}). 
Since 
$\overline{\beta}_{g,h}=e^{-nu/2}\beta_{g,h}$ 
we have for all 
$\psi\in C^{\infty}(\Sigma^gM)$ 
\begin{displaymath}
  D^{g,h}\psi=e^{-u/2}D^g(e^{-u/2}\psi).
\end{displaymath}

Let $f\in C^{\infty}(M,\mR)$ and let~$I\subset\mR$ 
be an open interval 
containing~$0$ such that for every~$t\in I$ 
the tensor field~$g_t:=g+tfg$ 
is a Riemannian metric on~$M$. 
Then for all~$\psi\in C^{\infty}(\Sigma^gM)$ we have 
\begin{displaymath}
  D^{g,g_t}\psi=(1+tf)^{-1/4}D^g((1+tf)^{-1/4}\psi)
\end{displaymath}
and therefore 
\begin{equation}
  \label{dirac_derivative_conform}
  \frac{d}{dt}D^{g,g_t}\big|_{t=0}\psi=
  -\frac{1}{2}fD^g\psi-\frac{1}{4}\grad^g(f)
  \cdot\psi,
\end{equation}
where the Clifford multiplication is taken with 
respect to the metric~$g$. 
Using this formula in (\ref{lambda_deriv_general}) 
and taking the real part we obtain 
\begin{equation}
\label{lambda_derivative_conform}
\frac{d\lambda_{j,t}}{dt}\big|_{t=0}
=-\frac{\lambda}{2}\int_M f|\psi_{j,0}|^2\,\dv^g,
\quad 1\leq j\leq d.
\end{equation}

\subsection{Further preliminaries}
\label{further_prelim_section}

In this section we briefly recall a 
trans\-ver\-sa\-lity 
theorem from differential topology and a 
unique continuation theorem 
for generalized Laplace operators acting on 
sections of a vector bundle. 

\begin{definition}
Let $f$: $Q\to N$ be a $C^1$ map between two 
manifolds. 
Let~$A\subset N$ be a submanifold. 
$f$ is called transverse to~$A$, if for 
all~$x\in Q$ with~$f(x)\in A$ 
we have
\begin{displaymath}
  T_{f(x)}A+\im(df|_x)=T_{f(x)}N.
\end{displaymath}
\end{definition}

We quote the following transversality theorem 
from \cite{hirsch:76}, \cite{uhlenbeck:76}. 

\begin{theorem}
\label{param_transvers}
Let $V$, $M$, $\Sigma$ be smooth manifolds and 
let~$A\subset\Sigma$ be a smooth submanifold. 
Let $F$: $V\to C^r(M,\Sigma)$ be a map, 
such that the evaluation map 
\begin{displaymath}
F^{ev}:\quad 
V\times M\to\Sigma,\quad 
(v,m)\mapsto F(v)(m)
\end{displaymath}
is $C^r$ and transverse to $A$. 
If 
\begin{displaymath}
  r>\max\{0,\dim M+\dim A-\dim\Sigma\},
\end{displaymath}
then the set of all $v\in V$, 
such that the map~$F(v)$ is transverse to~$A$, 
is residual and therefore dense in~$V$.
\end{theorem}

Let $(M,g)$ be a Riemannian manifold and 
let~$\Sigma$ be a vector bundle over~$M$ 
with a connection~$\nabla$. 
Then the connection Laplacian 
\begin{displaymath}
  \nabla^*\nabla:\quad C^\infty(\Sigma)\to C^\infty(\Sigma)
\end{displaymath}
is a linear elliptic differential operator of 
second order. 
In terms of a local~$g$-orthonormal 
frame~$(e_i)_{i=1}^n$ of~$TM$ 
it is given by 
\begin{equation}
  \label{laplacian_local}
  \nabla^*\nabla\psi=-\sum_{i=1}^n \nabla_{e_i}\nabla_{e_i}\psi
  +\sum_{i=1}^n \nabla_{\nabla_{e_i}e_i}\psi.
\end{equation}
We will use the following unique continuation theorem 
due to Aronszajn (\cite{aronszajn:57}, quoted from \cite{baer:97}).

\begin{theorem}
\label{unique_cont}
Let $(M,g)$ be a connected Riemannian manifold and 
let~$\Sigma$ be a vector 
bundle over~$M$ with a connection~$\nabla$. 
Let~$P$ be an operator of 
the form~$P=\nabla^*\nabla+P_1+P_0$ acting on sections 
of~$\Sigma$, where~$P_1$,~$P_0$ are differential 
operators of order~$1$ and~$0$ respectively. 
Let~$\psi$ be a solution to~$P\psi=0$. 
If there exists a point, at 
which~$\psi$ and all derivatives of~$\psi$ of 
any order vanish, then~$\psi$ vanishes identically.
\end{theorem}

If $(M,g,\Theta)$ is a closed Riemannian spin 
manifold, 
then from the connection~$\nabla^g$ on the spinor 
bundle one obtains 
a connection Laplacian~$\nabla^*\nabla$. 
Later we will apply the above theorem to the 
operator $P=(D^g)^2-\lambda^2$, 
$\lambda\in\mR\setminus\{0\}$. 
This is possible since 
for all $\psi\in C^{\infty}(\Sigma^gM)$ we have 
by the Schr\"odinger-Lichnerowicz formula 
\begin{equation}
  \label{schroed_lichn}
  (D^g)^2\psi=\nabla^*\nabla\psi+\frac{\scal^g}{4}\psi
\end{equation}
(see \cite{lawson.michelsohn:89} p.\,160), 
where~$\scal^g$ is the scalar curvature of $(M,g)$.

\section{Green's function for the Dirac operator}
\label{green_chapter}

\subsection{The Bourguignon-Gauduchon trivialization}
\label{bourg-triv_section}

Let $(M,g,\Theta)$ be a Riemannian spin manifold of 
dimension~$n$. 
In this section we explain a local 
trivialization of the 
spinor bundle~$\Sigma^gM$, which we will use 
in order to describe the expansion of Green's 
function for the Dirac operator. 
In the literature (e.\,g.\,\cite{ammann.humbert:05}) 
it is known as the Bourguignon-Gauduchon 
trivialization. 

Let $p\in M$, let~$U$ be an open neighborhood of~$p$ 
in~$M$ 
and let~$V$ be an open neighborhood of~$0$ in~$\mR^n$, 
such that there exists a local parametrization 
$\rho$:~$V\to U$ of~$M$ 
by Riemannian normal coordinates with~$\rho(0)=p$.
The spinor bundle over the Euclidean 
space~$(\mR^n,\geucl)$ with the unique spin structure 
will be denoted by~$\Sigma\mR^n$. 
The Clifford multiplications on~$\Sigma^gM$ and 
on~$\Sigma\mR^n$ will both be denoted by a 
dot~$\cdot$. 
It will be clear from the context which one of 
these two multiplications we mean. 

Let $x\in V$. Then there exists an endomorphism 
$G_x\in\End(T_xV)$, 
such that 
for all vectors $v$, $w\in T_xV$ we have
\begin{displaymath}
  (\rho^*g)(v,w)=\geucl(G_x v,w)
\end{displaymath}
and $G_x$ is $\geucl$-self-adjoint and positive 
definite. 
There is a unique positive definite 
endomorphism~$B_x\in\End(T_xV)$ such 
that~$B_x^2=G_x^{-1}$. 
If~$(E_i)_{i=1}^n$ is a 
$\geucl$-orthonormal basis of~$T_xV$, 
then~$(B_x E_i)_{i=1}^n$, 
is a~$\rho^*g$-orthonormal basis of~$T_xV$. 
Therefore the vectors~$d\rho|_x B_x E_i$, 
$1\leq i\leq n$, 
form a~$g$-orthonormal basis of~$T_{\rho(x)}M$. 
We assemble the maps~$B_x$ to obtain a vector bundle 
endomorphism~$B$ of~$TV$ and we define
\begin{displaymath}
  b:\quad TV\to TM|_U,\quad b=d\rho\circ B.
\end{displaymath}
From this we obtain an isomorphism of 
principal~$\SO(n)$-bundles
\begin{displaymath}
  P_{\SO}(V,\geucl)\to P_{\SO}(U,g),\quad (E_i)_{i=1}^n\mapsto(b(E_i))_{i=1}^n,
\end{displaymath}
which lifts to an isomorphism of 
principal~$\Spin(n)$-bundles
\begin{displaymath}
  c:\quad P_{\Spin}(V,\geucl)\to P_{\Spin}(U,g).
\end{displaymath}
We define 
\begin{displaymath}
  \beta:\quad \Sigma\mR^n|_V\to\Sigma^gM|_U,\quad [s,\sigma]\mapsto[c(s),\sigma].
\end{displaymath}
This gives an identification of the spinor bundles, 
which is a fiberwise isometry with respect to the bundle metrics on~$\Sigma\mR^n|_V$ 
and on~$\Sigma^gM|_U$. 
Furthermore for all~$X\in TV$ and for 
all~$\varphi\in\Sigma\mR^n|_V$ 
we have the formula  
$\beta(X\cdot\varphi)=b(X)\cdot\beta(\varphi)$. 
We obtain an isomorphism
\begin{displaymath}
A:\quad
C^{\infty}(\Sigma^gM|_U)\to 
C^{\infty}(\Sigma\mR^n|_V),\quad
\psi\mapsto \beta^{-1}\circ\psi\circ\rho,
\end{displaymath}
which sends a spinor on~$U$ to the corresponding 
spinor in the trivialization. 
Let~$\nabla^g$ resp.\,$\nabla$ denote the Levi Civita 
connections on~$(U,g)$ resp.\,on~$(V,\geucl)$ 
as well as their lifts to~$\Sigma^gM|_U$ 
and~$\Sigma\mR^n|_V$. 

The connection $\nabla^g$ on $\Sigma^gM|_U$ may be 
written as follows. 
Let~$(e_i)_{i=1}^n$ be a positively oriented 
local orthonormal frame of~$TM$ on~$U$. 
There is a locally defined section 
$s\in C^{\infty}(\rP_{\Spin}(M,g)|_U)$ such 
that~$(e_i)_{i=1}^n=\Theta\circ s$ on~$U$. 
Let~$(E_i)_{i=1}^N$ be the standard basis of~$\mC^N$, 
where~$N:=2^{[n/2]}$. 
The section~$s$ determines a local orthonormal 
frame~$(\psi_i)_{i=1}^N$ 
of~$\Sigma^gM|_U$ via 
\begin{displaymath}
\psi_i=[s,E_i]\in 
C^{\infty}(\Sigma^gM|_U),\quad i=1,...,N.
\end{displaymath}
We denote by~$\partial$ the locally defined flat connection with respect to the 
local frame~$(\psi_i)_{i=1}^N$, i.e. 
for~$h_1$,...,$h_N\in C^{\infty}(U,\mC)$ 
and~$X\in TM|_U$ we define
\begin{displaymath}
  \partial_X\big(\sum_{i=1}^N h_i\psi_i\big):=\sum_{i=1}^N X(h_i) \psi_i.
\end{displaymath}
Then for all $\psi\in C^{\infty}(\Sigma^gM|_U)$ 
we have 
\begin{displaymath}
\nabla^g_{e_i}\psi=
\partial_{e_i}\psi+\frac{1}{4}\,
\sum_{j,k=1}^n\widetilde{\Gamma}^k_{ij} 
e_j\cdot e_k\cdot\psi,
\end{displaymath}
where
\begin{displaymath}
  \widetilde{\Gamma}^k_{ij}:=g(\nabla^g_{e_i}e_j,e_k)
\end{displaymath}
(see \cite{lawson.michelsohn:89}, p.\,103, 110). 

In particular we can take the standard basis 
$(E_i)_{i=1}^n$ of~$\mR^n$ and we can 
put~$e_i:=b(E_i)$, $1\leq i\leq n$. 
We define the matrix coefficients~$B^j_i$ by the 
equation $B(E_i)=\sum_{j=1}^n B^j_i E_j$. 
It follows that
\begin{align}
\nonumber
A\nabla^g_{e_i}\psi&=
\nabla_{d\rho^{-1}(e_i)}A\psi+\frac{1}{4}\,
\sum_{j,k=1}^n\widetilde{\Gamma}^k_{ij} E_j
\cdot E_k\cdot A\psi\\
\nonumber
&=\nabla_{B(E_i)}A\psi+\frac{1}{4}\,\sum_{j,k=1}^n
\widetilde{\Gamma}^k_{ij} E_j\cdot E_k\cdot A\psi\\
\nonumber 
&=\nabla_{E_i}A\psi+\sum_{j=1}^n(B^j_i-\delta^j_i)
\nabla_{E_j}A\psi\\
\label{nabla_triv}
&{}+\frac{1}{4}\,\sum_{j,k=1}^n
\widetilde{\Gamma}^k_{ij} E_j\cdot E_k
\cdot A\psi.
\end{align}
Hence we obtain
\begin{align}
\nonumber
AD^g\psi&=
D^{\geucl}A\psi
+\sum_{i,j=1}^n(B^j_i-\delta^j_i)E_i
\cdot\nabla_{E_j}A\psi\\
\label{dirac_triv}
&{}+\frac{1}{4}\sum_{i,j,k=1}^n
\widetilde{\Gamma}^k_{ij} E_i\cdot E_j\cdot E_k
\cdot A\psi.
\end{align}
Let $\partial_j:=d\rho(E_j)$, $1\leq j\leq n$, be the 
coordinate vector fields of the normal coordinates. 
The Taylor expansion of the coefficient~$g_{ij}$ of 
the metric at~$0$ is given by
\begin{displaymath}
  g_{ij}(x)=\delta_{ij}+\frac{1}{3}\sum_{a,b=1}^n R_{iabj}(p)x_a x_b+O(|x|^3),
\end{displaymath}
where $R_{iabj}
=g(R(\partial_b,\partial_j)\partial_a,\partial_i)$ 
denotes the components of the Riemann curvature tensor 
(see e.\,g.\,\cite{lee.parker:87}, p.\,42, 61). 
Since~$(B^j_i)_{ij}=(g_{ij})_{ij}^{-1/2}$ it follows that
\begin{equation}
  \label{B-expansion}
  B^j_i(x)=\delta^j_i-\frac{1}{6}\sum_{a,b=1}^n R_{iabj}(p)x_a x_b +O(|x|^3).
\end{equation}
Since we have $\nabla^g_{\partial_k}\partial_r|_p=0$ for all $k,r$, we obtain 
\begin{equation}
  \label{gammatilde}
  \widetilde{\Gamma}^m_{kr}(\rho(x))=O(|x|)
\end{equation}
as $x\to0$ for all $m$, $k$, $r$.

\subsection{The Euclidean Dirac operator}

In this section we obtain pre\-images 
under the Dirac operator 
of certain spinors on~$\mR^n\setminus\{0\}$ with the 
Euclidean metric. 
The results will be useful for obtaining the expansion 
of Green's function for the Dirac operator on a closed 
spin manifold. 

\begin{definition}
\label{pkmi_def}
For $k\in\mR$, $m\in\mN$ and $i\in\{0,1\}$ we define 
the vector subspaces $P_{k,m,i}(\mR^n)$ of 
$C^{\infty}(\Sigma\mR^n|_{\mR^n\setminus\{0\}})$ 
as follows. For $k\neq0$ define 
\begin{align*}
  P_{k,m,0}(\mR^n)&:=\spann\Big\{x\mapsto x_{i_1}...x_{i_m}|x|^k\gamma\,
  \Big|\begin{array}{l}1\leq i_1,...,i_m\leq n,\\ \gamma\in\Sigma_n\textrm{ constant}\end{array}\Big\}\\
  P_{k,m,1}(\mR^n)&:=\spann\Big\{x\mapsto x_{i_1}...x_{i_m}|x|^k x\cdot\gamma\,
  \Big|\begin{array}{l}1\leq i_1,...,i_m\leq n,\\ \gamma\in\Sigma_n\textrm{ constant}\end{array}\Big\}
\end{align*}
and furthermore define 
\begin{align*}
P_{0,m,0}(\mR^n)&:=\spann\Big\{x\mapsto x_{i_1}...x_{i_m}\ln|x|\gamma\,
\Big|\begin{array}{l}1\leq i_1,...,i_m\leq n,\\ \gamma\in\Sigma_n\textrm{ constant}\end{array}\Big\}\\
P_{0,m,1}(\mR^n)&:=\spann\Big\{x\mapsto x_{i_1}...x_{i_m}(1-n\ln|x|) x\cdot\gamma\,
\Big|\hspace{-0.5em}\begin{array}{l}1\leq i_1,...,i_m\leq n,\hspace{-0.14em}\\ \gamma\in\Sigma_n\textrm{ constant}\end{array}
\hspace{-0.3em}\Big\}.
\end{align*}
\end{definition}

With these definitions we can prove the following 
proposition.

\begin{proposition}
\label{prop_invert_dirac}
For all $m\in\mN$, $k\in\mR$ with the 
properties~$-n\leq k$ 
and~$-n<k+m\leq0$ we have 
\begin{displaymath}
  P_{k,m,0}(\mR^n)\subset D^{\geucl}\Big(\sum_{j=1}^{[\frac{m+1}{2}]}P_{k+2j,m+1-2j,0}(\mR^n)+\sum_{j=0}^{[\frac{m}{2}]}P_{k+2j,m-2j,1}(\mR^n)\Big)
\end{displaymath}
For all $m\in\mN$, $k\in\mR$ with $-n\leq k$ and $-n<k+m+1\leq0$ we have 
\begin{displaymath}
  P_{k,m,1}(\mR^n)\subset D^{\geucl}\Big(\sum_{j=0}^{[\frac{m}{2}]}P_{k+2+2j,m-2j,0}(\mR^n)+\hspace{-0.05em}\sum_{j=1}^{[\frac{m+1}{2}]}P_{k+2j,m+1-2j,1}(\mR^n)\Big).
\end{displaymath}
\end{proposition}

\begin{proof}
We use induction on $m$. 
Let $m=0$ and let $\gamma$ be a constant spinor 
on~$(\mR^n,\geucl)$. 
We want to prove that 
$P_{k,0,0}(\mR^n)\subset D^{\geucl}(P_{k,0,1}(\mR^n))$ 
for all~$k$ with $-n<k\leq0$ and 
$P_{k,0,1}(\mR^n)\subset
D^{\geucl}(P_{k+2,0,0}(\mR^n))$ for all~$k$ 
with~$-n\leq k$ and $k+1\leq0$. 
One calculates easily that 
\begin{align*}
D^{\geucl}\big(-\frac{1}{n+k}|x|^kx\cdot\gamma\big)&=|x|^k\gamma,\quad k\neq-n\\
D^{\geucl}\big(\frac{1-n\ln|x|}{n^2}x\cdot\gamma\big)
&=\ln|x|\gamma,\\
D^{\geucl}\big(\frac{1}{k+2}|x|^{k+2}\gamma\big)
&=|x|^kx\cdot\gamma,\quad k\neq-2\\
D^{\geucl}(\ln|x|\gamma)&=|x|^{-2}x\cdot\gamma.
\end{align*}
Using the definition of $P_{k,0,i}(\mR^n)$ one 
obtains the assertion for $m=0$.

Let $m\geq1$ and assume that all the inclusions in the assertion hold for~$m-1$. 
Using the equation $E_i\cdot x=-2x_i-x\cdot E_i$ we find 
\begin{align*}
& D^{\geucl}(-x_{i_1}...x_{i_m}|x|^k x\cdot\gamma)\\
&=-\sum_{j=1}^m
x_{i_1}...\widehat{x_{i_j}}...x_{i_m}|x|^k
E_{i_j}\cdot x\cdot\gamma
-x_{i_1}...x_{i_m} D^{\geucl}(|x|^k x\cdot\gamma)\\
&=(2m+n+k)x_{i_1}...x_{i_m}|x|^k\gamma
+\sum_{j=1}^m
x_{i_1}...\widehat{x_{i_j}}...x_{i_m}|x|^k 
x\cdot E_{i_j}\cdot\gamma.
\end{align*}
Since $E_{i_j}\cdot\gamma$ is a parallel spinor the 
sum on the right hand side 
is contained in~$P_{k,m-1,1}(\mR^n)$. 
We apply the 
induction hypothesis and since~$2m+n+k\neq0$ we 
obtain the assertion for~$P_{k,m,0}(\mR^n)$. 
We define~$f_k(x):=\frac{1}{k}|x|^k$ for $k\neq0$ 
and $f_0(x):=\ln|x|$. 
Then we find 
\begin{align*}
& D^{\geucl}(x_{i_1}...x_{i_m}f_{k+2}(x)\gamma)\\
&=x_{i_1}...x_{i_m}|x|^k x\cdot\gamma
+\sum_{j=1}^m
x_{i_1}...\widehat{x_{i_j}}...x_{i_m}f_{k+2}(x)
E_{i_j}\cdot\gamma.
\end{align*}
The sum on the right hand side is in 
$P_{k+2,m-1,0}(\mR^n)$. 
Again we apply the induction hypothesis and we 
obtain the assertion for~$P_{k,m,1}(\mR^n)$.
\end{proof}

\subsection{Expansion of Green's function}

Let~$(M,g,\Theta)$ be a closed Riemannian spin 
manifold,~$n=\dim M$ and~$\lambda\in\mR$. 
Let $\pi_i$: $M\times M\to M$, $i=1,2$ be the 
projections. We define
\begin{displaymath}
\Sigma^gM\boxtimes\Sigma^gM^*
:=\pi_1^*\Sigma^gM\otimes(\pi_2^*\Sigma^gM)^*
\end{displaymath}
i.\,e.\,$\Sigma^gM\boxtimes\Sigma^gM^*$ is the 
vector bundle over~$M\times M$ 
whose fiber over the point $(x,y)\in M\times M$ is 
given by $\Hom_{\mC}(\Sigma^g_yM,\Sigma^g_xM)$. 
We define $\Delta:=\{(x,x)|x\in M\}$. 
In the following we will abbreviate
\begin{displaymath}
\int_{M\setminus\{p\}}
:=\lim_{\varepsilon\to0}
\int_{M\setminus B_{\varepsilon}(p)}.
\end{displaymath}

\begin{definition}
A smooth section $G^g_{\lambda}$: 
$M\times M\setminus\Delta
\to\Sigma^gM\boxtimes\Sigma^gM^*$ 
which is locally integrable on~$M\times M$ is called 
a Green's function for $D^g-\lambda$ 
if for all~$p\in M$ and for 
all~$\varphi\in\Sigma^g_pM$ 
the following holds: 
\begin{enumerate}
\item For all~$\psi\in\im(D^g-\lambda)$ we have
\begin{equation}
  \label{green_function1}
  \int_{M\setminus\{p\}} \big\langle
  (D^g-\lambda)\psi,G^g_{\lambda}(.,p)\varphi\big
  \rangle \dv^g
  =\big\langle \psi(p),\varphi \big\rangle.
 \end{equation}
\item For 
all~$\psi\in\ker(D^g-\lambda)$ we have
\begin{equation}
  \label{green_function2}
  \int_{M\setminus\{p\}} \big\langle
  \psi,G^g_{\lambda}(.,p)\varphi\big\rangle \dv^g=0.
\end{equation}
\end{enumerate}
\end{definition}
The smooth spinor $G^g_{\lambda}(.,p)\varphi$ on 
$M\setminus\{p\}$ will sometimes also 
be called Green's function for~$\varphi$. 
Let $P$: $ 
C^{\infty}(\Sigma^gM)\to\ker(D^g-\lambda)$ denote 
the~$L^2$-orthogonal projection. 
Then we have for all~$\psi\in C^{\infty}(\Sigma^gM)$ 
and for all~$\varphi\in\Sigma^g_pM$
\begin{equation}
  \label{greens_function}
  \int_{M\setminus\{p\}} \big\langle
  (D^g-\lambda)\psi,G^g_{\lambda}(.,p)\varphi\big
  \rangle \dv^g
  =\big\langle \psi(p)-P\psi(p),\varphi \big\rangle.
\end{equation}

On Euclidean space we define a Green's function as 
follows.

\begin{definition}
Let $(M,g)=(\mR^n,\geucl)$ with the unique spin 
structure. 
A smooth section $G^g_{\lambda}$: $M\times 
M\setminus\Delta\to\Sigma^gM\boxtimes\Sigma^gM^*$ 
which is locally integrable on~$M\times M$ is called 
a Green's function for~$D^g-\lambda$ 
if for all~$p\in M$, for all~$\varphi\in\Sigma^g_pM$ 
and for all~$\psi\in C^{\infty}(\Sigma^gM)$ 
with compact support the 
equation~(\ref{green_function1}) holds.
\end{definition}

Of course a Green's function for 
$D^{\geucl}-\lambda$ is not uniquely 
determined by this definition. 
We will explicitly write down a Green's function for 
$D^{\geucl}-\lambda$. 
First observe that for every spinor $\chi\in 
C^{\infty}(\Sigma\mR^n|_{\mR^n\setminus\{0\}})$ 
and for every $\lambda\in\mR$ the equation
\begin{displaymath}
  (D^{\geucl}-\lambda)(D^{\geucl}+\lambda)\chi
  =-\sum_{i=1}^n
  \nabla_{E_i}\nabla_{E_i}\chi-\lambda^2\chi
\end{displaymath}
holds on $\mR^n\setminus\{0\}$. 
Let $\gamma$ be a constant spinor on 
$(\mR^n,\geucl)$ 
and let~$g$ be a solution to the ordinary 
differential equation 
\begin{equation}
  \label{green_ode}
  g''(z)+\frac{n-1}{z}g'(z)+\lambda^2g(z)=-\delta_0,
\end{equation}
which is smooth on $(0,\infty)$. 
If we define $f$: $\mR^n\setminus\{0\}\to\mR$,
$f(x):=g(|x|)$, then 
the spinor 
$G^{\geucl}_{\lambda}(.,0)\gamma
:=(D^{\geucl}+\lambda)(f\gamma)$ 
is a Green's function.

In the following let $\Gamma$ denote the Gamma 
function and~$J_m$,~$Y_m$ 
the Bessel functions of the first and second kind 
for the parameter~$m\in\mR$. 
In the notation of \cite{abramowitz.stegun:64}, 
p.\,360 they are defined for~$z\in(0,\infty)$ by 
\begin{align*}
J_m(z)&=\frac{1}{2^m\Gamma(m+1)}z^m
\big(1+\sum_{k=1}^{\infty}a_k z^{2k}\big),
\quad m\in\mR,\\
Y_0(z)&=\frac{2}{\pi}
\big(\ln\big(\frac{z}{2}\big)+c\big)J_0(z)
+\sum_{k=1}^{\infty}b_k z^{2k},\\
Y_m(z)&=-\frac{2^m}{\pi}\Gamma(m)z^{-m}
\big(1+\sum_{k=1}^{\infty}c_k z^{2k}\big),\quad
m=\frac{1}{2}+k,k\in\mN, \\
Y_m(z)&=-\frac{2^m}{\pi}\Gamma(m)z^{-m}
\big(1+\sum_{k=1}^{\infty}d_k z^{2k}\big)
+\frac{2}{\pi}\ln\big(\frac{z}{2}\big)J_m(z),\,\,\,
m\in\mN\setminus\{0\},
\end{align*}
where $c$ is a real constant, the $a_k$, $b_k$, 
$c_k$, $d_k$ are real coefficients, 
the $a_k$, $c_k$, $d_k$ depend on $m$ 
and all the power series converge for all 
$z\in(0,\infty)$.
Let~$\omega_{n-1}$ be the volume 
of $S^{n-1}$ 
with the standard metric.

\begin{theorem}
\label{theorem_green_eucl}
Let $m:=\frac{n-2}{2}$. 
We define $f_{\lambda}$: $\mR^n\setminus\{0\}\to\mR$ 
as follows. 
For $\lambda\neq0$ and $n=2$ 
\begin{displaymath}
f_{\lambda}(x)
:=-\frac{1}{4}Y_0(|\lambda x|)
+\frac{\ln|\lambda|-\ln(2)+c}{2\pi}J_0(|\lambda x|),
\end{displaymath}
for $\lambda\neq0$ and odd $n\geq3$ 
\begin{displaymath}
f_{\lambda}(x):=
-
\frac{\pi|\lambda|^m}{2^m\Gamma(m)(n-2)\omega_{n-1}}
|x|^{-m}Y_m(|\lambda x|),
\end{displaymath}
for $\lambda\neq0$ and even $n\geq4$ 
\begin{align*}
f_{\lambda}(x)&:=
-\frac{\pi|\lambda|^m}{2^m\Gamma(m)(n-2)\omega_{n-1}}
|x|^{-m}\\
&\times\Big(Y_m(|\lambda x|)
-\frac{2(\ln|\lambda|-\ln(2))}{\pi}
J_m(|\lambda x|)\Big)
\end{align*}
and 
\begin{displaymath}
f_0(x):=-\frac{1}{2\pi}\ln|x|,\quad n=2,\qquad 
f_0(x):=\frac{1}{(n-2)\omega_{n-1}|x|^{n-2}},
\quad n\geq3.
\end{displaymath}
Then for every constant spinor $\gamma$ on $\mR^n$ a 
Green's function for $D^{\geucl}-\lambda$ is given by
\begin{displaymath}
G^{\geucl}_{\lambda}(x,0)\gamma=(D^{\geucl}+\lambda)
(f_{\lambda}\gamma)(x).
\end{displaymath}
\end{theorem}

\begin{corollary}
\label{coroll_green_asymp}
For every constant spinor $\gamma\in\Sigma_n$ there 
exists a Green's function 
of~$D^{\geucl}-\lambda$, which has the following 
form. For $n=2$ 
\begin{displaymath}
G^{\geucl}_{\lambda}(x,0)\gamma
=-\frac{1}{2\pi|x|}\frac{x}{|x|}\cdot\gamma
-\frac{\lambda}{2\pi}\ln|x|\gamma
+\ln|x|\vartheta_{\lambda}(x)+\zeta_{\lambda}(x),
\end{displaymath}
for odd $n\geq3$ 
\begin{displaymath}
G^{\geucl}_{\lambda}(x,0)\gamma
=-\frac{1}{\omega_{n-1}|x|^{n-1}}\frac{x}{|x|}
\cdot\gamma
+\frac{\lambda}{(n-2)\omega_{n-1}|x|^{n-2}}\gamma
+|x|^{2-n}\zeta_{\lambda}(x),
\end{displaymath}
for even $n\geq4$ 
\begin{align*}
G^{\geucl}_{\lambda}(x,0)\gamma
&=-\frac{1}{\omega_{n-1}|x|^{n-1}}\frac{x}{|x|}
\cdot\gamma
+\frac{\lambda}{(n-2)\omega_{n-1}|x|^{n-2}}\gamma
+|x|^{2-n}\zeta_{\lambda}(x)\\
&{}-\frac{\lambda^{n-1}}{2^{n-2}
\Gamma(\frac{n}{2})^2\omega_{n-1}}\ln|x|\gamma
+\ln|x|\vartheta_{\lambda}(x),
\end{align*}
where for every $n$ and for every~$\lambda$ the 
spinors~$\vartheta_{\lambda}$ and~$\zeta_{\lambda}$ 
extend smoothly to~$\mR^n$ and satisfy 
\begin{displaymath}
|\zeta_{\lambda}(x)|_{\geucl}=O(|x|),\quad
|\vartheta_{\lambda}(x)|_{\geucl}=O(|x|)\quad
\textrm{ as } x\to0
\end{displaymath}
and where for every $n$ and for every $x$ the 
spinors 
$\vartheta_{\lambda}(x),
\zeta_{\lambda}(x)\in\Sigma_n$ 
are power series in $\lambda$ with 
$\vartheta_0(x)=\zeta_0(x)=0$.
\end{corollary}

\begin{proof}[Proof of Corollary 
\ref{coroll_green_asymp}]
Using the definition of $f_{\lambda}$ we find for 
$n=2$
\begin{displaymath}
f_{\lambda}(x)=-\frac{1}{2\pi}\ln|x|
\big(1+\sum_{k=1}^{\infty}a_k|\lambda x|^{2k}\big)
-\frac{1}{4}\sum_{k=1}^{\infty}b_k|\lambda x|^{2k},
\end{displaymath}
for odd $n\geq3$
\begin{displaymath}
f_{\lambda}(x)=\frac{1}{(n-2)\omega_{n-1}|x|^{n-2}}
\big(1+\sum_{k=1}^{\infty}c_k|\lambda x|^{2k}\big)
\end{displaymath}
and for even $n\geq4$
\begin{align*}
f_{\lambda}(x)&=
\frac{1}{(n-2)\omega_{n-1}|x|^{n-2}}
\big(1+\sum_{k=1}^{\infty}d_k|\lambda x|^{2k}\big)\\
&{}-\frac{\lambda^{n-2}}{2^{n-2}(m!)^2\omega_{n-1}}
\ln|x|\big(1+\sum_{k=1}^{\infty}a_k|
\lambda x|^{2k}\big).
\end{align*}
The assertion follows.
\end{proof}

\begin{proof}[Proof of Theorem 
\ref{theorem_green_eucl}]
Let $f_{\lambda}$ be as above and write 
$f_{\lambda}(x)=g_{\lambda}(|x|)$ 
with $g_{\lambda}$: $(0,\infty)\to\mR$. 
Then $g_{\lambda}$ solves the equation 
(\ref{green_ode}).
It remains to show that 
$(D^{\geucl}+\lambda)(f_{\lambda}\gamma)$ satisfies 
(\ref{green_function1}). 
The calculation in the proof of 
Corollary~\ref{coroll_green_asymp} shows
\begin{displaymath}
(D^{\geucl}+\lambda)(f_{\lambda}\gamma)(x)=
-\frac{1}{\omega_{n-1}} 
\frac{x}{|x|^n}\cdot\gamma+\zeta(x),
\end{displaymath}
where $|\zeta(x)|_{\geucl}=o(|x|^{1-n})$ as $x\to0$. 
For any Riemannian spin manifold $(M,g,\Theta)$ with 
boundary~$\partial M$ 
and $\psi$, $\varphi$ compactly supported spinors we 
have
\begin{displaymath}
(D^g\psi,\varphi)_2-(\psi,D^g\varphi)_2
=\int_{\partial M}
\langle \nu\cdot\psi,\varphi\rangle\,dA,
\end{displaymath}
where $\nu$ is the outer unit normal vector field on 
$\partial M$ (see \cite{lawson.michelsohn:89}, p.\,115). 
We apply this equation to $(\mR^n\setminus 
B_{\varepsilon}(0),\geucl)$ and 
$\nu(x):=-\frac{x}{|x|}$ 
and we obtain the assertion. 
\end{proof}

\begin{definition}
For $m\in\mR$ we define 
\begin{displaymath}
P_m(\mR^n):=\sum_{r+s+t\geq m\atop r\geq-n} 
P_{r,s,t}(\mR^n)
+( C^{\infty}(\Sigma\mR^n|_{\mR^n\setminus\{0\}})
\cap C^{0}(\Sigma\mR^n)),
\end{displaymath}
where the second space on the right hand side is the 
space of all spinors which are smooth 
on $\mR^n\setminus\{0\}$ and have a continuous 
extension to $\mR^n$.
\end{definition}

\begin{remark}
Let $\vartheta\in P_m(\mR^n)$. 
Then we have $E_i\cdot\vartheta\in P_m(\mR^n)$ for 
all~$i\in\{1,...,n\}$. 
If $f\in C^{\infty}(\mR^n,\mC)$ then from 
Taylor's formula for~$f$ it follows 
that the spinor~$f\vartheta$ is in~$P_m(\mR^n)$.
\end{remark}

\begin{remark}
\label{remark_invert_dirac}
If $m>0$ then every spinor in $P_m(\mR^n)$ has a 
continuous 
extension to~$\mR^n$. 
Furthermore by Proposition \ref{prop_invert_dirac} 
it follows that for all~$m\in(-n,0]$ we have 
\begin{align*}
P_m(\mR^n)&=\sum_{r+s+t=m\atop r\geq-n}
P_{r,s,t}(\mR^n)+P_{m+1}(\mR^n)\\
&\subset D^{\geucl}(P_{m+1}(\mR^n))+P_{m+1}(\mR^n).
\end{align*}
\end{remark}

\begin{lemma}
\label{green_p_2-n_lemma}
Let $(M,g,\Theta)$ be a closed Riemannian spin 
manifold of dimension~$n$ and let~$\lambda\in\mR$. 
Let~$\gamma\in\Sigma_n$ be 
a constant spinor on~$\mR^n$. 
Then the spinor $G^{\geucl}_{\lambda}(.,0)\gamma$ 
is in~$P_{1-n}(\mR^n)$. 
Let the matrix coefficients~$B^j_i$ be defined as 
in~(\ref{dirac_triv}).  
Then for all~$i$ the spinor
\begin{displaymath}
x\mapsto\sum_{j=1}^n 
(B^j_i(x)-\delta^j_i)\nabla_{E_j}
G^{\geucl}_{\lambda}(x,0)\gamma
\end{displaymath}
is in $P_{2-n}(\mR^n)$.
\end{lemma}

\begin{proof}
The assertion for $G^{\geucl}_{\lambda}(.,0)\gamma$ 
follows immediately from 
Corollary~\ref{coroll_green_asymp}. 
Let~$f_{\lambda}$ be as in 
Theorem~\ref{theorem_green_eucl} and write 
$f_{\lambda}(x)=g_{\lambda}(|x|)$ 
with a func\-tion $g_{\lambda}$: $(0,\infty)\to\mR$. 
Then we have 
\begin{displaymath}
G^{\geucl}_{\lambda}(x,0)\gamma=
\frac{g_{\lambda}'(|x|)}{|x|}x\cdot\gamma
+\lambda g_{\lambda}(|x|)\gamma
\end{displaymath}
and for every $j\in\{1,...,n\}$ we have 
\begin{align*}
\nabla_{E_j}G^{\geucl}_{\lambda}(x,0)\gamma&=
\frac{g_{\lambda}''(|x|)x_j}{|x|^2}x\cdot\gamma
-\frac{g_{\lambda}'(|x|)x_j}{|x|^3}x\cdot\gamma
+\frac{g_{\lambda}'(|x|)}{|x|}E_j\cdot\gamma\\
&{}
+\lambda\frac{g_{\lambda}'(|x|)x_j}{|x|}\gamma.
\end{align*}
As the exponential map is a radial isometry, we 
have 
$\sum_{j=1}^n g_{ij}(x)x_j=x_i$ and thus 
$\sum_{j=1}^n B^j_i(x)x_j=x_i$ for every fixed $i$. 
It follows that 
\begin{displaymath}
\sum_{j=1}^n (B^j_i(x)-\delta^j_i)\nabla_{E_j}
G^{\geucl}_{\lambda}(x,0)\gamma
=\sum_{j=1}^n 
(B^j_i(x)-\delta^j_i)
\frac{g_{\lambda}'(|x|)}{|x|}E_j\cdot\gamma.
\end{displaymath}
Since we have $g_{\lambda}'(|x|)=O(|x|^{1-n})$ 
as~$x\to0$ the assertion now follows from the Taylor 
expansion~(\ref{B-expansion}) of $B^j_i(x)$. 
\end{proof}

Next we prove existence and uniqueness of Green's 
function for~$D^g-\lambda$ 
on a closed Riemannian spin manifold in such a way 
that we also obtain 
the expansion of Green's function at the 
singularity. 
The idea is to apply the equation~(\ref{dirac_triv}) 
for the Dirac operator 
in the trivialization to a Euclidean Green's 
function and then determine 
the correction terms. 
For~$\lambda=0$ this has been carried out 
in~\cite{ammann.humbert:05}, 
where for some technical steps Sobolev embeddings 
were used. 
We present a more simple argument using the 
preimages under the 
Dirac operator from 
Proposition~\ref{prop_invert_dirac}.

In the following for a fixed point $p\in M$ let 
$\rho$: $V\to U$ be a local 
parametrization of~$M$ by Riemannian normal 
coordinates, where the subset $U\subset M$ 
is an open neighborhood of~$p$, $V\subset\mR^n$ is 
an open neighborhood of~$0$ 
and~$\rho(0)=p$. 
Furthermore let 
\begin{displaymath}
\beta:\quad \Sigma\mR^n|_V\to\Sigma^gM|_U,\quad 
A:\quad
C^{\infty}(\Sigma^gM|_U)\to 
C^{\infty}(\Sigma\mR^n|_V) 
\end{displaymath}
denote the maps which send a spinor to its 
corresponding spinor in the 
Bour\-guig\-non-Gau\-du\-chon trivialization 
defined in Section~\ref{bourg-triv_section}.

\begin{theorem}
\label{theorem_green}
Let $(M,g,\Theta)$ be a closed $n$-dimensional 
Riemannian spin manifold, $p\in M$. 
For every $\varphi\in\Sigma^g_pM$ there exists a 
unique Green's function 
$G^g_{\lambda}(.,p)\varphi$. 
If $\gamma:=\beta^{-1}\varphi\in\Sigma_n$ is the 
constant spinor on~$\mR^n$ 
corresponding to~$\varphi$, 
then the first two terms 
of the expansion of $AG^g_{\lambda}(.,p)\varphi$ 
at~$0$ 
coincide with the first two terms of the expansion 
of~$G^{\geucl}_{\lambda}(.,0)\gamma$ given in 
Corollary~\ref{coroll_green_asymp}.
\end{theorem}

\begin{proof}
Let $\varepsilon>0$ such that 
$B_{2\varepsilon}(0)\subset V$ and let 
$\eta:$ $\mR^n\to[0,1]$ be a smooth function with 
$\supp(\eta)\subset B_{2\varepsilon}(0)$ 
and $\eta\equiv1$ on $B_{\varepsilon}(0)$. 
Then the spinor~$\Theta_1$ defined 
on~$\mR^n\setminus\{0\}$ 
by~$\Theta_1(x):=
\eta(x)G^{\geucl}_{\lambda}(x,0)\gamma$ 
is smooth on $\mR^n\setminus\{0\}$. 
For $r\in\{1,...,n\}$ we define smooth 
spinors~$\Phi_r$ on~$M\setminus\{p\}$ 
and~$\Theta_{r+1}$ on~$\mR^n\setminus\{0\}$ 
inductively as follows. 
For~$r=1$ we define 
\begin{displaymath}
  \Phi_1(q):=\left\{\begin{array}{ll}
  A^{-1}\Theta_1(q),&q\in U\setminus\{p\}\\ 
  0,&q\in M\setminus U\end{array}\right.
\end{displaymath}
and 
\begin{displaymath}
  \Theta_2(x):=\left\{\begin{array}{ll}
  A(D^g-\lambda)\Phi_1(x),&x\in V\setminus\{0\}\\ 
  0,&x\in \mR^n\setminus V\end{array}\right. .
\end{displaymath}
By the formula (\ref{dirac_triv}) for the Dirac 
operator in the trivialization we have 
on~$V\setminus\{0\}$
\begin{align*}
\Theta_2
&=(D^{\geucl}-\lambda)\Theta_1
+\sum_{i,j=1}^n(B^j_i-\delta^j_i)E_i\cdot
\nabla_{E_j}\Theta_1\\
&{}+\frac{1}{4}\sum_{i,j,k=1}^n
\widetilde{\Gamma}^k_{ij} E_i\cdot E_j\cdot E_k
\cdot \Theta_1.
\end{align*}
The first term vanishes on 
$B_{\varepsilon}(0)\setminus\{0\}$. 
It follows from 
the expansions of~$\widetilde{\Gamma}^k_{ij}$ 
and~$B^j_i-\delta^j_i$ 
in (\ref{B-expansion}), (\ref{gammatilde}) and from 
Lemma \ref{green_p_2-n_lemma} 
that $\Theta_2\in P_{2-n}(\mR^n)$.

Next let $r\in\{2,...,n\}$ and assume 
that~$\Phi_{r-1}$ and~$\Theta_r$ 
have already been defined. 
By induction hypothesis we may assume that we have 
$\Theta_r\in P_{r-n}(\mR^n)$. 
By Remark~\ref{remark_invert_dirac} there exists 
$\beta_{r+1}\in P_{r+1-n}(\mR^n)$ such 
that~$\Theta_r-(D^{\geucl}-\lambda)\beta_{r+1}\in 
P_{r+1-n}(\mR^n)$. 
We define~$\Phi_r$ and~$\Theta_{r+1}$ by 
\begin{displaymath}
  \Phi_r(q):=\left\{\begin{array}{ll}
  \Phi_{r-1}(q)-A^{-1}(\eta\beta_{r+1})(q),&q\in
  U\setminus\{p\}\\ 
  0,&q\in M\setminus U\end{array}\right.
\end{displaymath}
and 
\begin{displaymath}
  \Theta_{r+1}(x):=\left\{\begin{array}{ll}
  A(D^g-\lambda)\Phi_r(x),&x\in V\setminus\{0\}\\ 
  0,&x\in \mR^n\setminus V\end{array}\right. .
\end{displaymath}
By the formula (\ref{dirac_triv}) for the Dirac 
operator in the trivialization we have 
on~$B_{\varepsilon}(0)\setminus\{0\}$
\begin{align*}
\Theta_{r+1}&=
A(D^g-\lambda)\Phi_{r-1}-A(D^g-\lambda)A^{-1}
\beta_{r+1}\\
&=\Theta_r-(D^{\geucl}-\lambda)\beta_{r+1}
-\sum_{i,j=1}^n(B^j_i-\delta^j_i)E_i
\cdot\nabla_{E_j}\beta_{r+1}\\
&{}-\frac{1}{4}\sum_{i,j,k=1}^n
\widetilde{\Gamma}^k_{ij} E_i\cdot E_j\cdot E_k
\cdot \beta_{r+1}.
\end{align*}
Using the expansions of $\widetilde{\Gamma}^k_{ij}$ 
and $B^j_i-\delta^j_i$ 
in (\ref{B-expansion}), (\ref{gammatilde}) we 
conclude that we have 
$\Theta_{r+1}\in P_{r+1-n}(\mR^n)$. 

We see that~$\Theta_{n+1}$ has a continuous 
extension to~$\mR^n$ 
and we obtain a continuous extension~$\Psi$ of 
$(D^g-\lambda)\Phi_n$ to all of~$M$. 
Thus there exists
\begin{displaymath}
\Psi'\in
C^{\infty}(\Sigma^gM|_{M\setminus\{p\}})
\cap H^1(\Sigma^gM)
\end{displaymath}
such that $(D^g-\lambda)\Psi'=P\Psi-\Psi$. Define
\begin{displaymath}
  \Gamma:=\Phi_n+\Psi',\quad
  \Theta:=-\eta\beta_3-...-\eta\beta_{n+1}+A\Psi'.
\end{displaymath}
Then on $B_{\varepsilon}(0)\setminus\{0\}$ 
we have 
$A\Gamma=G^{\geucl}_{\lambda}(.,0)\gamma+\Theta$.

If $\psi_1,...,\psi_d$ is an $L^2$-orthonormal basis 
of $\ker(D^g-\lambda)$, then for 
every number~$i\in\{1,...,d\}$ 
the integral 
\begin{displaymath}
C_i:=\int_{M\setminus\{p\}}\langle\Gamma,\psi_i
\rangle\,\dv^g
\end{displaymath}
exists, since we have 
$|A\Gamma(x)|_{\geucl}=O(|x|^{1-n})$ as 
$x\to0$. 
Then 
\begin{displaymath}
  G^g_{\lambda}(.,p)\varphi:=\Gamma-
  \sum_{i=1}^{d}C_i\psi_i
\end{displaymath}
satisfies (\ref{green_function1}), 
(\ref{green_function2}) 
and thus is a Green's function. 
Uniqueness also follows from 
(\ref{green_function1}), (\ref{green_function2}). 
The statement on the expansion of 
$AG^g_{\lambda}(.,p)\varphi$ is obvious, since we 
have $\Theta\in P_{3-n}(\mR^n)$. 
\end{proof}

\section{Zero sets of eigenspinors}
\label{gen_eigenspinor_chapter}

\subsection{Eigenspinors in dimensions $2$ and $3$}
\label{gen_eigenspinor_section}

In this section we prove 
Theorem~\ref{main_theorem}. 
Assume that $n\in\{2,3\}$. 
Then there exists a quaternionic structure on 
the spinor bundle, 
i.\,e.\,a conjugate linear endomorphism~$J$ 
of~$\Sigma^gM$, 
which satisfies $J^2=-\id$. 
Furthermore~$J$ is parallel and commutes 
with Clifford multiplication 
(see e.\,g.\,\cite{friedrich:00}, p.\,33). 
It follows that~$J$ commutes with the Dirac 
operator and thus every eigenspace 
of~$D^g$ has even complex dimension. 
Therefore the following notation introduced 
by Dahl~(see \cite{dahl:03}) is useful. 

\begin{definition}
Let $n\in\{2,3\}$. An eigenvalue~$\lambda$ 
of~$D^g$ is called simple, 
if one has $\dim_{\mC}\ker(D^g-\lambda)=2$.
\end{definition}

If $\lambda$ is a simple eigenvalue of~$D^g$, 
then one can choose an~$L^2$-ortho\-nor\-mal 
basis of $\ker(D^g-\lambda)$ 
of the form $\{\psi,J\psi\}$.

For every $g\in R(M)$ we enumerate the nonzero 
eigenvalues of~$D^g$ in the following way
\begin{displaymath}
  ...\leq\lambda_{-2}(g)
  \leq\lambda_{-1}(g)
  <0
  <\lambda_1(g)
  \leq\lambda_2(g)\leq...\,.
\end{displaymath}
Here all the non-zero eigenvalues are repeated 
by half of their complex multiplicities, 
while $\dim\ker(D^g)\geq0$ is arbitrary. 
For $m\in\mN\setminus\{0\}$ we define 
\begin{align*}
S_m(M)&:=\{g\in R(M)\,|
\lambda_{-m}(g),...,\lambda_m(g)
\textrm{ are simple}\}\\
N_m(M)&:=\Big\{g\in R(M)\,\Big|
\begin{array}{l}
\textrm{all eigenspinors to }
\lambda_{-m}(g),...,\lambda_m(g)\\
\textrm{are nowhere zero}
\end{array}
\Big\}.
\end{align*}

\begin{remark}
Let $(M,\Theta)$ be a closed spin manifold of 
dimension~$2$ or~$3$. 
It is known that there exists a subset 
$S_0(M,\Theta)\subset R(M)$ 
which is open and dense in~$R(M)$
such that the map 
\begin{displaymath}
S_0(M,\Theta)\to\mN,\quad 
g\mapsto\dim_{\mC}\ker(D^g)
\end{displaymath}
is constant. 
Furthermore for $n=3$ this constant is~$0$ 
while for $n=2$ it is either~$0$ or~$2$ 
depending on the topology of~$M$ and on the spin 
structure~$\Theta$ 
(see \cite{maier:97}, \cite{ammann.dahl.humbert:11}). 
Dahl has shown that for $n\in\{2,3\}$ and 
for every $m\in\mN\setminus\{0\}$ the subset 
$S_0(M,\Theta)\cap S_m(M)$ is open and dense 
in~$R(M)$. 
His proof shows that for every 
$g\in R(M)$ and for every~$m\in\mN\setminus\{0\}$ 
the subset $S_m(M)\cap[g]$ is open and dense 
in~$[g]$. 
All the statements in this remark hold with 
respect to all $C^k$-topologies, $k\geq1$, 
on $R(M)$. 
\end{remark}

Since for $n=3$ for all $g\in S_0(M,\Theta)$ 
we have $\dim_{\mC}\ker(D^g)=0$, 
it is sufficient for the proof of 
Theorem~\ref{main_theorem} 
in both cases $n\in\{2,3\}$
to consider non-harmonic eigenspinors. 
More precisely it is 
sufficient to prove the following theorem. 

\begin{theorem}
\label{theorem_no_zeros}
Let $M$ be a closed connected spin manifold of 
dimension~$2$ or~$3$ and 
let~$m\in\mN\setminus\{0\}$. 
Then $S_0(M,\Theta)\cap N_m(M)$ is open and dense 
in~$R(M)$. 
Furthermore for every~$g\in R(M)$ the subset 
$N_m(M)\cap[g]$ is open and dense in~$[g]$. 
\end{theorem}

In order to show density we will use 
Theorem~\ref{param_transvers} 
for families of spinors para\-me\-tri\-zed by 
Riemannian metrics. 
At first one might want to use 
the infinite-dimensional 
version of Theorem~\ref{param_transvers}, 
where~$V$ is equal to the space of 
all~$C^k$-metrics on~$M$ for some 
fixed $k\geq1$ (see~\cite{uhlenbeck:76}).  
However for non-smooth metrics~$g$ the 
coefficients of the 
Dirac operator~$D^g$ in a local chart and a local 
trivialization are not smooth by~(\ref{dirac_triv}) 
and thus we cannot expect that the eigenspinors 
are smooth 
and that we can apply Theorem~\ref{unique_cont}. 
Therefore we will use Theorem~\ref{param_transvers} 
with~$V$ equal to a finite-dimensional manifold 
contained in~$[g]$ of the form 
\begin{displaymath}
  V_{f_1...f_s}:=\Big\{\Big(1+\sum_{i=1}^s 
  t_i f_i\Big)g\,\Big|t_1,...,t_s\in\mR\Big\}
  \cap R(M),
\end{displaymath}
where $s\in\mN\setminus\{0\}$ and 
$f_1,...,f_s\in C^{\infty}(M,\mR)$.

Our first aim is to construct a map, 
which assigns to a Riemannian metric~$h$ 
an eigenspinor of~$D^{g,h}$ in a continuous way, 
where $g\in R(M)$ is fixed.

\begin{lemma}
\label{lemma_F_psi}
Let $m\in\mN\setminus\{0\}$ 
and let $g\in S_0(M,\Theta)\cap S_m(M)$. 
There exists an open neighborhood $V\subset R(M)$ 
of~$g$ such that for every number 
$i\in\{-m,...,m\}\setminus\{0\}$ 
and for every~$L^2$-normalized 
eigenspinor~$\psi$ of~$D^g$ corresponding 
to~$\lambda_i(g)$ there exists 
a map~$F_{\psi}$: $V\to C^{\infty}(\Sigma^gM)$ 
with the following properties: 
\begin{enumerate}
\item For every $h\in V$ the spinor~$F_{\psi}(h)$ 
is an eigenspinor of~$D^{g,h}$ corresponding 
to~$\lambda_i(h)$. 
\item The map $F^{ev}_{\psi}$: $V\times M\to\Sigma^gM$ 
defined by $F^{ev}_{\psi}(h,x):=F_{\psi}(h)(x)$ 
is continuous. 
\item For all functions 
$f_1,...,f_s\in C^{\infty}(M,\mR)$ 
the restriction of the map 
$F^{ev}_{\psi}$ to~$(V_{f_1...f_s}\cap V)\times M$ 
is a $C^1$-map.
\end{enumerate}
\end{lemma}

\begin{proof}
Let $h\in R(M)$. 
For every $t\in[0,1]$ the tensor 
$g_t:=g+t(h-g)$ is a Riemannian metric on~$M$. 
Then by Lemma \ref{rellich_lemma} there exist 
real analytic functions 
\begin{displaymath}
[0,1]\ni t\mapsto\lambda^h_{i,t}\in\mR,
\quad i\in\{-m,...,m\}\setminus\{0\},
\end{displaymath}
such that for every 
$i\in\{-m,...,m\}\setminus\{0\}$ 
and for every $t\in[0,1]$ the 
number~$\lambda^h_{i,t}$ 
is an eigenvalue of~$D^{g,g_t}$ and such that 
for every~$i$ we have 
$\lambda^h_{i,0}=\lambda_i(g)$. 
There exists an open neighborhood~$V$ 
of $g$ which is contained in $S_0(M,\Theta)$ 
such that for every~$h\in V$, 
for every number $i\in\{-m,...,m\}\setminus\{0\}$ 
and for every $t\in[0,1]$ we have 
\begin{equation}
\label{F_psi_proof}
\lambda^h_{i,t}=\lambda_i(g_t),
\end{equation}
i.\,e.\,the eigenvalue functions are nowhere 
zero on $[0,1]$ and their graphs do not 
intersect. 
This follows from the continuity of the 
eigenvalues of the Dirac operator in the 
$C^1$-topology 
(see e.\,g.\,Proposition 7.1 in~\cite{baer:96a}) 
together with the fact that for 
every $h\in S_0(M,\Theta)$ 
we have $\dim\ker(D^h)=\dim\ker(D^g)$.

Next let $h\in V$, let 
$i\in\{-m,...,m\}\setminus\{0\}$ 
and let $\psi$ be an eigenspinor of~$D^g$ 
corresponding 
to $\lambda:=\lambda_i(g)$. 
By Lemma~\ref{rellich_lemma} there exists a 
real analytic family $t\mapsto\lambda_t$, 
$t\in[0,1]$, 
of eigenvalues of~$D^{g,g_t}$ such that 
$\lambda_0=\lambda$ 
and there exists a real analytic family 
$t\mapsto\chi_t\in C^{\infty}(\Sigma^gM)$, 
$t\in[0,1]$, of~$L^2$-normalized eigenspinors 
of $D^{g,g_t}$ corresponding to $\lambda_t$. 
If $t\mapsto\zeta_t$, 
$t\in[0,1]$, is another such family
of eigenspinors, then 
there exist analytic functions 
$a,b$: $[0,1]\to\mC$ such that for all $t\in[0,1]$ 
we have $\zeta_t=a(t)\chi_t+b(t)J\chi_t$ 
and $|a(t)|^2+|b(t)|^2=1$. 
In order to fix a family of eigenspinors we 
apply a unitary transformation to~$\chi_t$ 
such that for every~$t$ small enough 
its component in $\ker(D^g-\lambda)$ 
becomes a positive multiple of~$\psi$. 
More precisely for every $t\in[0,1]$ 
we define $f(t):=(\chi_t,\psi)_{L^2}$ 
and $g(t):=(\chi_t,J\psi)_{L^2}$ 
and for every $t\in[0,1]$ such that 
$|f(t)|^2+|g(t)|^2>0$ we define 
\begin{displaymath}
\psi_t:=
\frac{\overline{f(t)}\chi_t-g(t)J\chi_t}
{(|f(t)|^2+|g(t)|^2)^{1/2}}.
\end{displaymath}
Then for all $t\in[0,1]$ such that 
$|f(t)|^2+|g(t)|^2>0$ the spinor~$\psi_t$ 
is an eigenspinor of~$D^{g,g_t}$ corresponding 
to~$\lambda_t$ and we have $\psi_0=\psi$. 
By this definition we have fixed a family 
of eigenspinors. 
After possibly shrinking~$V$ we may assume 
that for all metrics $h\in V$ the family 
$t\mapsto\psi_t$ is defined for all $t\in[0,1]$. 
We define $F_{\psi}(h):=\psi_1$. 
By~(\ref{F_psi_proof}) this eigenspinor 
of~$D^{g,h}$ corresponds to~$\lambda_i(h)$. 
If $\xi\in\ker(D^g-\lambda)$ is 
another~$L^2$-normalized eigenspinor 
then there exist 
$c,d\in\mC$ such that $|c|^2+|d|^2=1$ and 
$\xi=c\psi+dJ\psi$. 
It follows that the family 
$t\mapsto\xi_t:=c\psi_t+dJ\psi_t$ is defined 
for all $t\in[0,1]$. 
Thus for every~$L^2$-normalized eigenspinor 
$\xi\in\ker(D^g-\lambda)$ and for every $h\in V$ 
we can define~$F_{\xi}(h)$. 
For every $i\in\{-m,...,m\}\setminus\{0\}$ 
we obtain an open neighborhood $V\subset R(M)$ 
of~$g$ as above 
and after taking the intersection of these 
open neighborhoods we may assume that~$V$ is 
independent of~$i$. 
The first assertion follows. 

In order to show the second assertion 
let $h\in V$ and let~$(k_j)_{j\in\mN}$ 
be a sequence in~$V$ such that 
$k_j\to h$ as $j\to\infty$. 
Let $i\in\{-m,...,m\}\setminus\{0\}$, 
let~$\psi$ be an eigenspinor of~$D^g$ 
corresponding to~$\lambda_i(g)$. 
We define the~$L^2$-orthogonal projections 
\begin{align*}
P_j:&\quad  
C^{\infty}(\Sigma^gM)\to
\ker(D^{g,k_j}-\lambda_i(k_j)),\quad j\in\mN\\
P:&\quad C^{\infty}(\Sigma^gM)\to
\ker(D^{g,h}-\lambda_i(h)).
\end{align*}
We regard the operators $D^{g,h}$, 
$D^{g,k_j}$, $j\in\mN$, as closed operators 
on the space~$C^0(\Sigma^gM)$ with domain 
$C^1(\Sigma^gM)$. 
Then by Theorem~IV.3.16 in~\cite{kato:95} 
we get $P_j\to P$ with respect to the norm 
of bounded linear operators on~$C^0(\Sigma^gM)$. 
In particular for all $j\in\mN$ large enough 
the spinor 
$\alpha_j:=\frac{P_j(F_{\psi}(h))}
{\|P_j(F_{\psi}(h))\|_{L^2}}$ 
is well defined and we have 
$\|\alpha_j-F_{\psi}(h)\|_{C^0(\Sigma^gM)}\to0$ 
as~$j\to\infty$. 
By definition of $F_{\psi}$ we have for 
every $j\in\mN$: 
\begin{displaymath}
F_{\psi}(k_j)=
\frac{(\psi,\alpha_j)_{L^2}\alpha_j-
(\alpha_j,J\psi)_{L^2}J\alpha_j}
{(|(\psi,\alpha_j)_{L^2}|^2+
|(J\psi,\alpha_j)_{L^2}|^2)^{1/2}}.
\end{displaymath}
Let $x\in M$ and let $(x_j)_{j\in\mN}$ be a 
sequence in~$M$ such that $x_j\to x$ 
as~$j\to\infty$. 
It follows that 
$F_{\psi}^{ev}(k_j,x_j)\to F_{\psi}^{ev}(h,x)$ 
as $j\to\infty$. 
The second assertion follows. 

In order to prove the third assertion let 
$h,k\in V_{f_1...f_s}\cap V$ and let $x\in M$. 
For~$t\in[0,1]$ we define $h_t:=h+t(k-h)$ 
and $\psi_t:=F_{\psi}(h_t)$. 
The family $t\mapsto\psi_t$ is real analytic. 
In order to see this we use that 
by Lemma \ref{rellich_lemma} there exists 
a real analytic family 
$t\mapsto\alpha_t\in C^{\infty}(\Sigma^hM)$ 
of $L^2$-normalized eigenspinors of $D^{h,h_t}$. 
For all~$t\in[0,1]$ we have 
$\overline{\beta}_{h,h_t}^{-1}
=\overline{\beta}_{h_t,h}$ and thus 
\begin{displaymath}
\overline{\beta}_{h_t,g}
\overline{\beta}_{h,h_t}
D^{h,h_t}=D^{g,h_t}
\overline{\beta}_{h_t,g}
\overline{\beta}_{h,h_t}.
\end{displaymath}
Thus the family 
$t\mapsto
\overline{\beta}_{h_t,g}
\overline{\beta}_{h,h_t}\alpha_t
\in C^{\infty}(\Sigma^gM)$ is a real 
analytic family of~$L^2$-normalized 
eigenspinors of~$D^{g,h_t}$. 
The family $t\mapsto\psi_t$ is obtained 
from this family by a 
unitary transformation as above and therefore 
is real analytic. 
Thus the derivative of~$F_{\psi}^{ev}$ at~$(h,x)$ in 
the direction~$k-h$ exists and is given 
by~$\frac{d\psi_t}{dt}|_{t=0}(x)$. 
In order to show continuity of the derivative 
let~$(k_j)_{j\in\mN}$ be a sequence in 
$V_{f_1...f_s}\cap V$ 
such that $k_j\to h$ with respect to the 
$C^1$-topology as $j\to\infty$. 
After deleting finitely many of the~$k_j$ 
we may assume that for every~$j\in\mN$ and for every 
$t\in[0,1]$ the tensor field $k_{j,t}:=k_j+t(k-h)$ 
is a Riemannian metric on~$M$. 
Then we write 
$\psi^{k_j}_t:=
F_{\psi}(k_{j,t})$, 
$h_t:=h+t(k-h)$ and 
$\psi^h_t:=F_{\psi}(h_t)$. 
For all~$t\in[0,1]$ and for 
all~$j\in\mN$ we have 
\begin{align*}
\Re(\psi^{k_j}_t,i\psi)_{L^2}=0,
\quad 
(\psi^{k_j}_t,J\psi)_{L^2}=0,\quad 
(\psi^{k_j}_t,\psi^{k_j}_t)_{L^2}=1\\
\Re(\psi^h_t,i\psi)_{L^2}=0,\quad 
(\psi^h_t,J\psi)_{L^2}=0,\quad 
(\psi^h_t,\psi^h_t)_{L^2}=1.
\end{align*}
By taking the derivative at $t=0$ 
we obtain that for all $j\in\mN$ 
\begin{equation}
\label{F_psi_proof4}
\frac{d\psi^{k_j}_t}{dt}\big|_{t=0}
\in\ker(D^g-\lambda)^{\perp},\quad 
\frac{d\psi^h_t}{dt}\big|_{t=0}
\in\ker(D^g-\lambda)^{\perp}.
\end{equation}
For all $t\in[0,1]$ and for all 
$j\in\mN$ we have 
\begin{displaymath}
(D^{g,k_{j,t}}-\lambda^{k_j}_t)\psi^{k_j}_t=0,
\quad
(D^{g,h_t}-\lambda^h_t)\psi^h_t=0.
\end{displaymath}
By taking the derivative at $t=0$ we obtain 
for all $j\in\mN$ 
\begin{equation}
\label{F_psi_proof5}
(D^g-\lambda)\frac{d\psi^{k_j}_t}{dt}\big|_{t=0}
=-\Big(\frac{d}{dt}D^{g,k_{j,t}}\big|_{t=0}
-\frac{d\lambda^{k_j}_t}{dt}\big|_{t=0}
\Big)\psi
\end{equation}
and analogously for~$h$ instead of~$k_j$. 
Since the formulas~(\ref{dirac_derivative_conform}) 
and~(\ref{lambda_derivative_conform}) for the 
derivatives of the Dirac operator and of the 
eigenvalue contain derivatives of the metric of 
order at most~$1$ and since $k_j\to h$ 
in the $C^1$-topology we obtain 
\begin{displaymath}
\Big\|
\Big(\frac{d}{dt}D^{g,h_{t}}\big|_{t=0}
-\frac{d\lambda^{h}_t}{dt}\big|_{t=0}
\Big)\psi
-\Big(\frac{d}{dt}D^{g,k_{j,t}}\big|_{t=0}
-\frac{d\lambda^{k_j}_t}{dt}\big|_{t=0}
\Big)\psi
\Big\|_{C^0(\Sigma^gM)}\to0
\end{displaymath}
as $j\to\infty$. 
By (\ref{F_psi_proof5}) we conclude 
\begin{equation}
\label{F_psi_proof6}
\Big\|(D^g-\lambda)
\Big(
\frac{d\psi^{k_j}_t}{dt}\big|_{t=0}
-\frac{d\psi^h_t}{dt}\big|_{t=0}
\Big)\Big\|_
{C^0(\Sigma^gM)}\to0
\end{equation}
as $j\to\infty$. 
For $i=0,1$ we define 
\begin{displaymath}
U_i:=C^i(\Sigma^gM)\cap(\ker(D^g-\lambda))^\perp.
\end{displaymath}
and equip this space with the $C^i$-norm. 
Then the operator 
\begin{displaymath}
(D^g-\lambda)|_{U_1}:\quad U_1\to U_0
\end{displaymath}
is bounded and bijective. 
Thus its inverse is also bounded and 
from (\ref{F_psi_proof4}) 
and (\ref{F_psi_proof6}) 
we obtain 
\begin{displaymath}
\Big\|
\frac{d\psi^{k_j}_t}{dt}\big|_{t=0}
-\frac{d\psi^h_t}{dt}\big|_{t=0}
\Big\|_
{C^1(\Sigma^gM)}\to0
\end{displaymath}
as $j\to\infty$. 
Let $(x_j)_{j\in\mN}$ be a sequence in $M$ 
with $x_j\to x$ as $j\to\infty$. 
We obtain 
$\frac{d\psi^{k_j}_t}{dt}|_{t=0}(x_j)
\to\frac{d\psi^h_t}{dt}|_{t=0}(x)$ 
and the third assertion follows. 
\end{proof}

\begin{remark}
Let $m\in\mN\setminus\{0\}$ and 
assume that $g\in S_m(M)$, not necessarily 
$g\in S_0(M,\Theta)$. 
There exists an open neighborhood $V\subset[g]$ 
of~$g$ in~$[g]$ with analogous properties 
as the neighborhood~$V$ in Lemma~\ref{lemma_F_psi}. 
Namely for the proof of~(\ref{F_psi_proof}) 
we use that for all $h\in[g]$ we have 
$\dim\ker(D^g)=\dim\ker(D^h)$. 
\end{remark}

\begin{lemma}
\label{lemma_Nm_open}
Let $M$ be a closed spin manifold of dimension~$2$ 
or~$3$ and let~$m\in\mN\setminus\{0\}$. 
Then $S_0(M,\Theta)\cap N_m(M)$ is open in~$R(M)$. 
Furthermore for every $g\in R(M)$ the subset 
$N_m(M)\cap[g]$ is open in~$[g]$. 
\end{lemma}

\begin{proof}
Let $g\in S_0(M,\Theta)\cap N_m(M)$. 
Then $g\in S_0(M,\Theta)\cap S_m(M)$. 
We choose an open neighborhood $V\subset R(M)$ 
of~$h$ as in Lemma~\ref{lemma_F_psi}. 
Then for every number $i\in\{-m,...,m\}\setminus\{0\}$ 
we choose an~$L^2$-normalized 
eigenspinor~$\psi_i$ of~$D^h$ corresponding 
to~$\lambda_i(h)$ and we define the map 
$F_{\psi_i}$: $V\to C^{\infty}(\Sigma^hM)$ 
as in Lemma~\ref{lemma_F_psi}. 
Since~$F_{\psi_i}^{ev}$ is continuous 
and the complement of 
the zero section is open in~$\Sigma^gM$ 
it follows that for every~$i$ the subset 
\begin{displaymath}
A_{\psi_i}:=\{h\in V|\,F_{\psi_i}(h)
\textrm{ is nowhere zero}\}
\end{displaymath}
is an open neighborhood of~$g$ in~$R(M)$. 
We observe that for every~$i$ we have 
\begin{displaymath}
A_{\psi_i}
\subset
\{h\in V|\textrm{ all eigenspinors to }\lambda_i(h)
\textrm{ are nowhere zero}\}
\end{displaymath}
since $\lambda_i(h)$ 
is a simple eigenvalue of~$D^{g,h}$ and thus 
the zero sets 
of all eigenspinors corresponding to~$\lambda_i(h)$ 
coincide. 
Then the intersection of the 
subsets~$A_{\psi_i}$, 
where $i\in\{-m,...,m\}\setminus\{0\}$, 
is an open neighborhood of~$g$ in~$R(M)$ and 
is contained in $S_0(M,\Theta)\cap N_m(M)$. 
The first assertion follows. 
The proof of the second assertion is analogous. 
\end{proof}

The strategy for the proof of density in 
Theorem~\ref{theorem_no_zeros} is based 
on the following remark.

\begin{remark}
\label{remark_transverse}
Let $A\subset\Sigma^gM$ be the zero section.
The dimension of the total space $\Sigma^gM$ of 
the spinor bundle is 
\begin{displaymath}
  \dim\Sigma^gM=n+2^{1+[n/2]}>2n=\dim M+\dim A,
\end{displaymath}
and thus a map $f$: $M\to\Sigma^gM$ is transverse 
to~$A$ if and only if we have~$f^{-1}(A)=\emptyset$.
\end{remark}

Therefore in order to prove that 
$S_0(M,\Theta)\cap N_m(M)$ is dense in~$R(M)$ 
and that for every $g\in R(M)$ the subset 
$N_m(M)\cap[g]$ is dense in~$[g]$ 
we would like to apply 
Theorem~\ref{param_transvers}. 
Our aim is then to show that a suitable 
restriction of 
the map~$F_{\psi}^{ev}$ defined as in 
Lemma~\ref{lemma_F_psi} 
is transverse to the zero section.

Let $p\in M$ with $\psi(p)=0$. 
We have a canonical decomposition of the 
tangent space
\begin{displaymath}
  T_{\psi(p)}\Sigma^gM\cong \Sigma^g_pM\oplus T_pM
\end{displaymath}
and thus
\begin{displaymath}
  dF_{\psi}^{ev}|_{(g,p)}:\quad
  T_gV_{f_1...f_s}\oplus T_pM\to\Sigma^g_pM
  \oplus T_pM.
\end{displaymath}
For a given $h\in V_{f_1...f_s}$ we will 
write $g_t=g+t(h-g)$ and 
\begin{displaymath}
\psi_t:=F_{\psi}(g_t),\quad 
\frac{d\psi_t(x)}{dt}\big|_{t=0}
:=\pi_1(dF_{\psi}^{ev}|_{(g,x)}(h-g,0)),
\end{displaymath}
where~$\pi_1$ is the projection onto the first 
summand. 
Then it follows that
\begin{equation}
  \label{derivative_of_eigenspinor_equation}
  0=\big(\frac{d}{dt}D^{g,g_t}\big|_{t=0}
  -\frac{d\lambda_t}{dt}\big|_{t=0}\big)\psi
  +\big(D^g-\lambda\big)\frac{d\psi_t}{dt}
  \big|_{t=0}.
\end{equation}

\begin{remark}
\label{remark_real_basis}
Let $\varphi\in\Sigma^g_pM$ and $X$, $Y\in T_pM$. If 
one polarizes the identity
\begin{displaymath}
  \langle X\cdot\varphi,X\cdot\varphi\rangle
  =g(X,X)\langle\varphi,\varphi\rangle,
\end{displaymath}
then one obtains
\begin{equation}
  \label{Xgamma_Ygamma}
  \Re\langle X\cdot\varphi,Y\cdot\varphi\rangle
  =g(X,Y)\langle\varphi,\varphi\rangle.
\end{equation}
Since Clifford multiplication with vectors is 
antisymmetric, we obtain 
$\Re\langle X\cdot\varphi,\varphi\rangle=0$. 
Let $\varphi\neq0$ and let 
$(e_i)_{i=1}^n$ be an orthonormal basis of~$T_pM$. 
It follows that for $n=2$ the spinors
\begin{displaymath}
  \varphi,\,e_1\cdot\varphi,\,e_2\cdot\varphi,\,
  e_1\cdot e_2\cdot\varphi
\end{displaymath}
form an orthogonal basis of $\Sigma^g_pM$ with respect 
to the real scalar product $\Re\langle.,.\rangle$. 
Similarly for $n=3$ the spinors 
\begin{displaymath}
\varphi,\,e_1\cdot\varphi,\,e_2\cdot\varphi,\,
e_3\cdot\varphi
\end{displaymath}
form an orthogonal basis of $\Sigma^g_pM$ with respect 
to $\Re\langle.,.\rangle$.
\end{remark}

The following rather long lemma is the most important step in showing 
that a suitable restriction of $F^{ev}_{\psi}$ is transverse to the 
zero section.

\begin{lemma}
\label{lemma_transverse}
Let $M$ be a closed connected spin manifold 
of dimension~$2$ or~$3$. 
Let $m\in\mN\setminus\{0\}$, 
let $g\in S_m(M)$ and let 
$\lambda\in\{\lambda_{-m}(g),...,
\lambda_m(g)\}$. 
Let~$\psi$ be an~$L^2$-normalized 
eigenspinor of~$D^g$ corresponding to~$\lambda$ 
and let $p\in M$ with $\psi(p)=0$. 
Then there exist $f_1,...,f_4\in C^{\infty}(M,\mR)$ 
such that the map 
$F_{\psi}^{ev}$: $V_{f_1...f_4}\times M\to\Sigma^gM$ 
satisfies
\begin{displaymath}
  \pi_1(dF_{\psi}^{ev}|_{(g,p)}(T_gV_{f_1...f_4}
  \oplus\{0\}))=\Sigma^g_pM.
\end{displaymath}
\end{lemma}

\begin{proof}
Assume that the claim is wrong. 
Then there is $\varphi\in\Sigma^g_pM\setminus\{0\}$ 
such that for all $f\in C^{\infty}(M,\mR)$ we have 
\begin{displaymath}
  0=\Re\langle\pi_1(dF_{\psi}^{ev}|_{(g,p)}(fg,0)),\varphi\rangle
  =\Re\langle\frac{d\psi_t}{dt}\big|_{t=0}(p),\varphi\rangle.
\end{displaymath}
From the formula (\ref{greens_function}) for 
Green's function it follows that
\begin{displaymath}
  0=\Re \int_{M\setminus\{p\}} \big\langle \big(D^g-\lambda\big)\frac{d\psi_t}{dt}\big|_{t=0}, G^g_{\lambda}(.,p)\varphi \big\rangle \dv^g 
  +\Re \big\langle P\big(\frac{d\psi_t}{dt}\big|_{t=0}\big)(p), \varphi \big\rangle.
\end{displaymath}
Since $\lambda$ is a simple eigenvalue, all spinors in $\ker(D^g-\lambda)$ 
vanish at $p$. 
Thus the last term vanishes. 
By~(\ref{green_function2}) 
and~(\ref{derivative_of_eigenspinor_equation}) 
we have
\begin{align*}
0&=-\Re \int_{M\setminus\{p\}}
\big\langle \big(\frac{d}{dt}D^{g,g_t}\big|_{t=0}
-\frac{d\lambda_t}{dt}\big|_{t=0}\big)\psi,
G^g_{\lambda}(.,p)\varphi \big\rangle \dv^g\\
&=-\Re \int_{M\setminus\{p\}} \big\langle
\frac{d}{dt}D^{g,g_t}\big|_{t=0}\psi,
G^g_{\lambda}(.,p)\varphi \big\rangle \dv^g
\end{align*}
for all $f\in C^{\infty}(M,\mR)$. 
If we use the formula~(\ref{dirac_derivative_conform}) 
for the derivative of the Dirac operator and
\begin{displaymath}
  \grad^g(f)\cdot\psi=(D^g-\lambda)(f\psi)
\end{displaymath}
it follows that
\begin{align*}
0&=\frac{1}{2}\,\Re \int_{M\setminus\{p\}} 
\lambda f \big\langle \psi,
G^g_{\lambda}(.,p)\varphi \big\rangle \dv^g\\
&{}+\frac{1}{4}\,\Re \int_{M\setminus\{p\}}
\big\langle (D^g-\lambda)(f\psi),
G^g_{\lambda}(.,p)\varphi \big\rangle \dv^g.
\end{align*}
Using the definition of Green's function and using 
that all spinors in $\ker(D^g-\lambda)$ 
vanish at~$p$, we find that
\begin{align*}
0&=\frac{1}{2}\,\Re \int_{M\setminus\{p\}} 
\lambda f \big\langle \psi,
G^g_{\lambda}(.,p)\varphi \big\rangle \dv^g
+\frac{1}{4}\,\Re \big\langle
(f\psi)(p)-P(f\psi)(p),\varphi \big\rangle\\
&=\frac{1}{2}\,\Re \int_{M\setminus\{p\}} 
\lambda f \big\langle\psi, G^g_{\lambda}(.,p)
\varphi \big\rangle \dv^g
\end{align*}
for all $f\in C^{\infty}(M,\mR)$. 
Since~$\lambda\neq0$ it follows that 
$\Re\langle\psi,G^g_{\lambda}(.,p)\varphi\rangle$ vanishes identically on~$M\setminus\{p\}$. 

Our aim is now to conclude that all the derivatives 
of~$\psi$ at the point~$p$ vanish. 
Then by Theorem \ref{unique_cont} it follows 
that~$\psi$ is identically zero, which is a 
contradiction. 
In order to show this we choose a local 
parametrization $\rho$:~$V\to U$ of~$M$ 
by Riemannian normal coordinates, where $U\subset M$ 
is an open neighborhood of~$p$, 
$V\subset\mR^n$ is an open neighborhood of~$0$ 
and~$\rho(0)=p$. 
Furthermore let 
\begin{displaymath}
\beta:\quad \Sigma\mR^n|_V\to\Sigma^gM|_U,\quad 
A:\quad
C^{\infty}(\Sigma^gM|_U)\to 
C^{\infty}(\Sigma\mR^n|_V)
\end{displaymath}
denote the maps which send a spinor to its 
corresponding spinor in the 
Bour\-guig\-non-Gau\-du\-chon trivialization 
defined in Section~\ref{bourg-triv_section}.
We show by induction that $\nabla^r A\psi(0)=0$ for 
all~$r\in\mN$, where~$\nabla$ 
denotes the covariant derivative on~$\Sigma\mR^n$. 
The case~$r=0$ is clear. 

Let $r\geq1$ and assume that we have 
$\nabla^s A\psi(0)=0$ for all $s\leq r-1$. 
Let $(E_i)_{i=1}^n$ be the standard basis of~$\mR^n$. 
First consider the case~$n=2$. 
In the Bourguignon-Gauduchon trivialization we have
\begin{displaymath}
  A\psi(x)=\frac{1}{r!}\sum_{j_1,...,j_r=1}^2 x_{j_1}...x_{j_r}\nabla_{E_{j_1}}...\nabla_{E_{j_r}}A\psi(0)+O(|x|^{r+1})
\end{displaymath}
by Taylor's formula and
\begin{displaymath}
A (G^g_{\lambda}(.,p)\varphi)(x)
=-\frac{1}{2\pi|x|^2}x\cdot\gamma
-\frac{\lambda}{2\pi}\ln|x|\gamma+O(|x|^0)
\end{displaymath}
by Theorem \ref{theorem_green}, where 
$\gamma:=\beta^{-1}\varphi\in\Sigma_n$ is the constant 
spinor on $\mR^n$ corresponding to $\varphi$. 
It follows that
\begin{align}
\nonumber 
0&=-2\pi r!|x|^2 \Re\langle 
A (G^g_{\lambda}(.,p)\varphi)(x),A\psi(x)\rangle\\
\nonumber 
&=\sum_{i,j_1,...,j_r=1}^2 x_ix_{j_1}...x_{j_r}
\Re\langle
E_i\cdot\gamma,\nabla_{E_{j_1}}...\nabla_{E_{j_r}}
A\psi(0) \rangle\\
\label{taylor_dim2}
&{}+\lambda\sum_{j_1,...,j_r=1}^2
x_{j_1}...x_{j_r}|x|^2\ln|x|
\Re\langle\gamma,\nabla_{E_{j_1}}...\nabla_{E_{j_r}}
A\psi(0)
\rangle+O(|x|^{r+2}). 
\end{align}
For all $k_1,k_2\in\mN$ with $k_1+k_2=r+1$ 
the coefficient 
of~$x_1^{k_1}x_2^{k_2}$ must be zero. 
This coefficient is obtained from the first sum 
on the right hand side and it is equal to
\begin{displaymath}
\Re\langle E_1\cdot\gamma,
\nabla_{E_1}^{k_1-1}\nabla_{E_2}^{k_2}A\psi(0)
\rangle\frac{k_1r!}{k_1!k_2!}
+\Re\langle E_2\cdot\gamma,
\nabla_{E_1}^{k_1}\nabla_{E_2}^{k_2-1}A\psi(0)
\rangle\frac{k_2r!}{k_1!k_2!}.
\end{displaymath}
Thus if we write 
$x_1^{k_1}x_2^{k_2}=x_ix_{j_1}...x_{j_r}$ 
with $i,j_1,...,j_r\in\{1,2\}$, we obtain
\begin{align}
\nonumber
0&=\frac{r!}{k_1!k_2!}\Big(\Re\langle
E_i\cdot\gamma,\nabla_{E_{j_1}}...\nabla_{E_{j_r}}
A\psi(0) \rangle\\
\label{taylor_dim2_1}
&{}+\sum_{s=1}^r \Re\langle
E_{j_s}\cdot\gamma,\nabla_{E_i}\nabla_{E_{j_1}}...
\widehat{\nabla_{E_{j_s}}}...\nabla_{E_{j_r}}
A\psi(0) \rangle\Big)
\end{align}
for all $i,j_1,...,j_r\in\{1,2\}$, 
where the hat means that the operator is left out. 
For all $k_1,k_2\in\mN$ with $k_1+k_2=r$ the 
coefficient of~$x_1^{k_1+2}x_2^{k_2}\ln|x|$ 
in~(\ref{taylor_dim2}) must be zero. 
This coefficient is obtained from the second sum on 
the right hand side and it is equal to 
\begin{displaymath} 
\frac{\lambda r!}{k_1!k_2!}
\Re\langle\gamma,\nabla_{E_1}^{k_1}
\nabla_{E_2}^{k_2}A\psi(0)
\rangle
+\frac{\lambda r!k_2(k_2-1)}{k_1!k_2!}
\Re\langle\gamma,\nabla_{E_1}^{k_1+2}
\nabla_{E_2}^{k_2-2}
A\psi(0)\rangle.
\end{displaymath}
Using induction on $k_2$ and using that 
$\lambda\neq0$ 
we obtain for all $k_1,k_2\in\mN$ with~$k_1+k_2=r$ 
that 
\begin{displaymath}
\Re\langle\gamma,\nabla_{E_1}^{k_1}\nabla_{E_2}^{k_2}
A\psi(0)\rangle=0
\end{displaymath}
and therefore
\begin{equation}
  \label{taylor_dim2_2}
  0=\Re\langle\gamma,\nabla_{E_{j_1}}...\nabla_{E_{j_r}}A\psi(0)\rangle
\end{equation}
for all $j_1,...,j_r,i\in\{1,2\}$. 
By the equation~(\ref{taylor_dim2_2}) and by 
Remark~\ref{remark_real_basis} 
there exist $a_{j_1,...,j_r,k}$, $b_{j_1,...,j_r}\in\mR$ such that 
\begin{displaymath}
  \nabla_{E_{j_1}}...\nabla_{E_{j_r}}A\psi(0)=\sum_{k=1}^2 a_{j_1,...,j_r,k}E_k\cdot\gamma
  +b_{j_1,...,j_r}E_1\cdot E_2\cdot\gamma.
\end{displaymath}
Observe that the coefficients $a_{j_1,...,j_r,k}$ are 
symmetric in the first~$r$ indices. 
We insert this into~(\ref{taylor_dim2_1}) and we 
obtain
\begin{equation}
  \label{taylor_dim2_3}
  0=a_{j_1,...,j_r,i}+\sum_{k=1}^r a_{i,j_1,...,\widehat{j_k},...,j_r,j_k}
\end{equation}
for all $j_1,...,j_r,i\in\{1,2\}$. 
On the other hand since $\psi\in\ker(D^g-\lambda)$ 
we find using the induction hypothesis 
\begin{align}
\nonumber
0&=\lambda\nabla_{E_{j_1}}...\nabla_{E_{j_{r-1}}}
A\psi(0)\\
\nonumber
&=\nabla_{E_{j_1}}...\nabla_{E_{j_{r-1}}}
\sum_{i=1}^2 E_i\cdot\nabla_{E_i}A\psi(0)\\
\nonumber
&=\sum_{i,k=1}^2 a_{j_1,...,j_{r-1},i,k}E_i
\cdot E_k\cdot\gamma
+\sum_{i=1}^2 b_{j_1,...,j_{r-1},i}E_i\cdot E_1
\cdot E_2\cdot\gamma\\
\nonumber
&=-(a_{j_1,...,j_{r-1},1,1}+a_{j_1,...,j_{r-1},2,2})
\gamma\\
\nonumber
&{}+(a_{j_1,...,j_{r-1},1,2}
-a_{j_1,...,j_{r-1},2,1})E_1\cdot E_2\cdot\gamma\\
\label{taylor_dim2_4}
&{}+b_{j_1,...,j_{r-1},2}E_1\cdot\gamma
-b_{j_1,...,j_{r-1},1}E_2\cdot\gamma
\end{align}
for all $j_1,...,j_{r-1}\in\{1,2\}$. 
Thus $b_{j_1,...,j_r}=0$ 
for all $j_1,...,j_r\in\{1,2\}$. 
Next consider $a_{j_1,...,j_r,i}$ with fixed 
$j_1,...,j_r,i\in\{1,2\}$. 
If we have $j_k=i$ for all $k\in\{1,...,r\}$, then 
by~(\ref{taylor_dim2_3}) we know that 
$a_{j_1,...,j_r,i}=0$. 
If there exists~$k$ such that~$j_k\neq i$ it follows 
from the coefficient of $E_1\cdot E_2\cdot\gamma$ 
in~(\ref{taylor_dim2_4}) that
\begin{displaymath}
a_{i,j_1,...\widehat{j_k}...,j_r,j_k}
=a_{j_1,...,j_r,i}.
\end{displaymath}
Again (\ref{taylor_dim2_3}) yields 
$a_{j_1,...,j_r,i}=0$. 
We conclude that all $a_{j_1,...,j_r,i}$ vanish and 
that~$\nabla^r A\psi(0)=0$. 
This proves the assertion in the case~$n=2$.

Next consider $n=3$. 
In the Bourguignon-Gauduchon trivialization we have
\begin{align*}
A\psi(x)&=\frac{1}{r!}\sum_{j_1,...,j_r=1}^3
x_{j_1}...x_{j_r}\nabla_{E_{j_1}}...\nabla_{E_{j_r}}
A\psi(0)\\
&{}+\frac{1}{(r+1)!}\sum_{j_1,...,j_r,i=1}^3
x_{j_1}...x_{j_r}x_i\nabla_{E_{j_1}}...
\nabla_{E_{j_r}}\nabla_{E_i}A\psi(0) 
+o(|x|^{r+1})
\end{align*}
by Taylor's formula and
\begin{displaymath}
A(G^g_{\lambda}(.,p)\varphi)(x)
=-\frac{1}{4\pi|x|^3}x\cdot\gamma
+\frac{\lambda}{4\pi|x|}\gamma+o(|x|^{-s})
\end{displaymath}
for every $s>0$ by Theorem \ref{theorem_green}, where $\gamma$ is as above. It follows that
\begin{align}
\nonumber
0&=-4\pi r!|x|^3 \Re\langle 
A (G^g_{\lambda}(.,p)\varphi)(x),A\psi(x)\rangle\\
\nonumber
&=\sum_{i,j_1,...,j_r=1}^3 x_{j_1}...x_{j_r}x_i
\Re\langle
E_i\cdot\gamma,\nabla_{E_{j_1}}...\nabla_{E_{j_r}}
A\psi(0) \rangle\\
\nonumber
&{}+\frac{1}{r+1}\sum_{i,j_1,...,j_r,m=1}^3
x_{j_1}...x_{j_r}x_i x_m
\Re\langle E_i\cdot\gamma,\nabla_{E_{j_1}}...
\nabla_{E_{j_r}}\nabla_{E_m}A\psi(0) \rangle\\
\label{taylor_dim3}
&{}-\lambda\sum_{j_1,...,j_r=1}^3
x_{j_1}...x_{j_r}|x|^2
\Re\langle\gamma,\nabla_{E_{j_1}}...\nabla_{E_{j_r}}
A\psi(0)\rangle+o(|x|^{r+2}).
\end{align}
Analogously to the case $n=2$ we obtain 
from the first term on the right hand side
\begin{align}
\nonumber
0&=\Re\langle
E_i\cdot\gamma,\nabla_{E_{j_1}}...\nabla_{E_{j_r}}
A\psi(0) \rangle\\
\label{taylor_dim3_1}
&{}+\sum_{s=1}^r \Re\langle
E_{j_s}\cdot\gamma,\nabla_{E_i}\nabla_{E_{j_1}}...
\widehat{\nabla_{E_{j_s}}}...\nabla_{E_{j_r}}
A\psi(0) \rangle.
\end{align}
for all $j_1,...,j_r,i\in\{1,2,3\}$, where the hat 
means that the operator is left out. 
Our next aim is to obtain an analogue 
of~(\ref{taylor_dim2_2}) from the second and third 
term on the right hand side of~(\ref{taylor_dim3}). 
It is more difficult than in the case~$n=2$, 
since derivatives of both orders~$r$ and~$r+1$ appear. 
The equation~(\ref{dirac_triv}) reads
\begin{align*}
\lambda A\psi&=D^{\geucl}A\psi
+\sum_{i,j=1}^3
(B^j_i-\delta^j_i)E_i\cdot\nabla_{E_j}A\psi\\
&{}+\frac{1}{4}\sum_{i,j,k=1}^3
\widetilde{\Gamma}^k_{ij} E_i\cdot E_j\cdot E_k
\cdot A\psi.
\end{align*}
Using (\ref{B-expansion}), (\ref{gammatilde}) and that $|A\psi(x)|_{\geucl}=O(|x|^r)$ as $x\to0$ we find
\begin{displaymath}
  \lambda A\psi=D^{\geucl}A\psi+O(|x|^{r+1})
\end{displaymath}
and therefore
\begin{equation}
  \label{taylor_dim3_2}
  \nabla_{E_{j_1}}...\nabla_{E_{j_r}}D^{\geucl}
  A\psi(0)
  =\lambda\nabla_{E_{j_1}}...\nabla_{E_{j_r}}A\psi(0)
\end{equation}
for all $j_1,...,j_r\in\{1,2,3\}$. 
Using the equation (\ref{nabla_triv}) we find
\begin{displaymath}
  A\nabla^g_{e_i}\psi=\nabla_{E_i}A\psi
  +O(|x|^{r+1}),\quad
  A\nabla^g_{e_i}\nabla^g_{e_j}\psi=\nabla_{E_i}
  \nabla_{E_j}A\psi+O(|x|^r)
\end{displaymath}
for all $i,j\in\{1,2,3\}$. Since by definition $d\rho|_x^{-1}(e_i)=E_i+O(|x|^2)$ 
the second term in the local formula (\ref{laplacian_local}) for $\nabla^*\nabla$ 
vanishes at $p$ and therefore we find
\begin{displaymath}
  A\nabla^*\nabla\psi=-\sum_{i=1}^3 \nabla_{E_i}\nabla_{E_i}A\psi+O(|x|^r).
\end{displaymath}
From the Schr\"odinger-Lichnerowicz formula (\ref{schroed_lichn}) it follows that
\begin{displaymath}
  \lambda^2 A\psi-\frac{\scal}{4}A\psi
  =-\sum_{i=1}^3\nabla_{E_i}\nabla_{E_i}A\psi
  +O(|x|^r)
\end{displaymath}
and thus
\begin{equation}
  \label{taylor_dim3_3}
  \nabla_{E_{j_1}}...\nabla_{E_{j_{r-1}}}
  \sum_{i=1}^3\nabla_{E_i}\nabla_{E_i}A\psi(0)=0
\end{equation}
for all $j_1,...,j_{r-1}\in\{1,2,3\}$. 
Now recall the second and third term on the 
right hand side of (\ref{taylor_dim3})
\begin{align*}
0&=\frac{1}{r+1}\sum_{j_1,...,j_r,i,m=1}^3
x_{j_1}...x_{j_r}x_i x_m
\Re\langle
E_i\cdot\gamma,\nabla_{E_{j_1}}...\nabla_{E_{j_r}}
\nabla_{E_m}A\psi(0) \rangle\\
&{}-\lambda\sum_{j_1,...,j_r=1}^3
x_{j_1}...x_{j_r}|x|^2
\Re\langle\gamma,\nabla_{E_{j_1}}...\nabla_{E_{j_r}}
A\psi(0)\rangle
\end{align*}
and let $k_1$, $k_2$, $k_3\in\mN$ such that 
$k_1+k_2+k_3=r$. 
Then from the coefficient of 
$x_1^{k_1+2}x_2^{k_2}x_3^{k_3}$ we find
\begin{align*}
0&=\Re\langle
E_1\cdot\gamma,\nabla_{E_1}^{k_1}\nabla_{E_2}^{k_2}
\nabla_{E_3}^{k_3}\nabla_{E_1}A\psi(0)\rangle
\frac{r!}{(k_1+1)!k_2!k_3!} \hspace{2.8 em} (I)\\
&{}+\Re\langle
E_2\cdot\gamma,\nabla_{E_1}^{k_1}
\nabla_{E_2}^{k_2-1}\nabla_{E_3}^{k_3}
\nabla_{E_1}^2 A\psi(0)\rangle
\frac{r!k_2}{(k_1+2)!k_2!k_3!} \hspace{2 em} (III)\\
&{}+\Re\langle
E_3\cdot\gamma,\nabla_{E_1}^{k_1}\nabla_{E_2}^{k_2}
\nabla_{E_3}^{k_3-1}\nabla_{E_1}^2 A\psi(0)\rangle
\frac{r!k_3}{(k_1+2)!k_2!k_3!} \hspace{2 em}(IV)\\
&{}-\lambda\Re\langle\gamma,\nabla_{E_1}^{k_1}
\nabla_{E_2}^{k_2}\nabla_{E_3}^{k_3}A\psi(0)\rangle
\frac{r!}{k_1!k_2!k_3!} \hspace{8.7 em} (V)\\
&{}-\lambda\Re\langle\gamma,\nabla_{E_1}^{k_1}
\nabla_{E_2}^{k_2-2}\nabla_{E_3}^{k_3}
\nabla_{E_1}^2 A\psi(0)\rangle
\frac{r!k_2(k_2-1)}{(k_1+2)!k_2!k_3!} 
\hspace{3.4 em} (VII)\\
&{}-\lambda\Re\langle\gamma,\nabla_{E_1}^{k_1}
\nabla_{E_2}^{k_2}\nabla_{E_3}^{k_3-2}
\nabla_{E_1}^2 A\psi(0)\rangle
\frac{r!k_3(k_3-1)}{(k_1+2)!k_2!k_3!} 
\hspace{3.2 em} (VIII).\\
\end{align*}
From the coefficient of $x_1^{k_1}x_2^{k_2+2}x_3^{k_3}$ we find
\begin{align*}
0&=\Re\langle E_1\cdot\gamma,\nabla_{E_1}^{k_1-1}
\nabla_{E_2}^{k_2}\nabla_{E_3}^{k_3}
\nabla_{E_2}^2 A\psi(0)\rangle
\frac{r!k_1}{k_1!(k_2+2)!k_3!} \hspace{2 em} (II)\\
&{}+\Re\langle
E_2\cdot\gamma,\nabla_{E_1}^{k_1}\nabla_{E_2}^{k_2}
\nabla_{E_3}^{k_3}\nabla_{E_2} A\psi(0)\rangle
\frac{r!}{k_1!(k_2+1)!k_3!} \hspace{3.1 em} (I)\\
&{}+\Re\langle
E_3\cdot\gamma,\nabla_{E_1}^{k_1}\nabla_{E_2}^{k_2}
\nabla_{E_3}^{k_3-1}\nabla_{E_2}^2 A\psi(0)\rangle
\frac{r!k_3}{k_1!(k_2+2)!k_3!} \hspace{2.2 em} (IV)\\
&{}-\lambda\Re\langle\gamma,\nabla_{E_1}^{k_1-2}
\nabla_{E_2}^{k_2}\nabla_{E_3}^{k_3}\nabla_{E_2}^2
A\psi(0)\rangle
\frac{r!k_1(k_1-1)}{k_1!(k_2+2)!k_3!} 
\hspace{3.5 em} (VI)\\
&{}-\lambda\Re\langle\gamma,\nabla_{E_1}^{k_1}
\nabla_{E_2}^{k_2}\nabla_{E_3}^{k_3} A\psi(0)\rangle
\frac{r!}{k_1!k_2!k_3!} \hspace{8.7 em} (V)\\
&{}-\lambda\Re\langle\gamma,\nabla_{E_1}^{k_1}
\nabla_{E_2}^{k_2}\nabla_{E_3}^{k_3-2}
\nabla_{E_2}^2 A\psi(0)\rangle
\frac{r!k_3(k_3-1)}{k_1!(k_2+2)!k_3!} 
\hspace{3.6 em} (VIII).\\
\end{align*}
From the coefficient of $x_1^{k_1}x_2^{k_2}x_3^{k_3+2}$ we find
\begin{align*}
0&=\Re\langle
E_1\cdot\gamma,\nabla_{E_1}^{k_1-1}
\nabla_{E_2}^{k_2}\nabla_{E_3}^{k_3}
\nabla_{E_3}^2 A\psi(0)\rangle
\frac{r!k_1}{k_1!k_2!(k_3+2)!} \hspace{2 em} (II)\\
&{}+\Re\langle
E_2\cdot\gamma,\nabla_{E_1}^{k_1}
\nabla_{E_2}^{k_2-1}\nabla_{E_3}^{k_3}
\nabla_{E_3}^2 A\psi(0)\rangle
\frac{r!k_2}{k_1!k_2!(k_3+2)!} 
\hspace{2.2 em} (III)\\
&{}+\Re\langle 
E_3\cdot\gamma,\nabla_{E_1}^{k_1}\nabla_{E_2}^{k_2}
\nabla_{E_3}^{k_3}\nabla_{E_3} A\psi(0)\rangle
\frac{r!}{k_1!k_2!(k_3+1)!} \hspace{3.2 em} (I)\\
&{}-\lambda\Re\langle\gamma,\nabla_{E_1}^{k_1-2}
\nabla_{E_2}^{k_2}\nabla_{E_3}^{k_3}\nabla_{E_3}^2
A\psi(0)\rangle
\frac{r!k_1(k_1-1)}{k_1!k_2!(k_3+2)!} 
\hspace{3.5 em} (VI)\\
&{}-\lambda\Re\langle\gamma,\nabla_{E_1}^{k_1}
\nabla_{E_2}^{k_2-2}\nabla_{E_3}^{k_3}
\nabla_{E_3}^2 A\psi(0)\rangle
\frac{r!k_2(k_2-1)}{k_1!k_2!(k_3+2)!} 
\hspace{3.6 em} (VII)\\
&{}-\lambda\Re\langle\gamma,\nabla_{E_1}^{k_1}
\nabla_{E_2}^{k_2}\nabla_{E_3}^{k_3} A\psi(0)
\rangle
\frac{r!}{k_1!k_2!k_3!} \hspace{8.7 em} (V).
\end{align*}
We multiply the first equation with $\frac{(k_1+2)!k_2!k_3!}{r!}$, the second equation with $\frac{k_1!(k_2+2)!k_3!}{r!}$ 
and the third equation with $\frac{k_1!k_2!(k_3+2)!}{r!}$ and then add the multiplied equations. 
If we consider the lines with the same Roman numbers separately and use (\ref{taylor_dim3_2}), (\ref{taylor_dim3_3}), 
then we find
\begin{align*}
0&=-2\lambda\Re\langle\gamma,\nabla_{E_1}^{k_1}
\nabla_{E_2}^{k_2}\nabla_{E_3}^{k_3}A\psi(0)\rangle\\
&{}+\Re\langle
E_1\cdot\gamma,\nabla_{E_1}^{k_1}\nabla_{E_2}^{k_2}
\nabla_{E_3}^{k_3}\nabla_{E_1}A\psi(0)\rangle k_1\\
&{}+\Re\langle
E_2\cdot\gamma,\nabla_{E_1}^{k_1}\nabla_{E_2}^{k_2}
\nabla_{E_3}^{k_3}\nabla_{E_2}A\psi(0)\rangle k_2\\
&{}+\Re\langle 
E_3\cdot\gamma,\nabla_{E_1}^{k_1}\nabla_{E_2}^{k_2}
\nabla_{E_3}^{k_3}\nabla_{E_3}A\psi(0)\rangle
k_3\hspace{7 em} (I)\\
&{}-\Re\langle 
E_1\cdot\gamma,\nabla_{E_1}^{k_1}\nabla_{E_2}^{k_2}
\nabla_{E_3}^{k_3}\nabla_{E_1}A\psi(0)\rangle
k_1\hspace{7 em} (II)\\
&{}-\Re\langle 
E_2\cdot\gamma,\nabla_{E_1}^{k_1}\nabla_{E_2}^{k_2}
\nabla_{E_3}^{k_3}\nabla_{E_2}A\psi(0)\rangle
k_2\hspace{7 em} (III)\\
&{}-\Re\langle 
E_3\cdot\gamma,\nabla_{E_1}^{k_1}\nabla_{E_2}^{k_2}
\nabla_{E_3}^{k_3}\nabla_{E_3}A\psi(0)\rangle
k_3\hspace{7 em} (IV)\\
&{}-\lambda\Re\langle\gamma,\nabla_{E_1}^{k_1}
\nabla_{E_2}^{k_2}\nabla_{E_3}^{k_3}A\psi(0)\rangle
\sum_{i=1}^3(k_i+2)(k_i+1)\hspace{3 em} (V)\\
&{}+\lambda\Re\langle\gamma,\nabla_{E_1}^{k_1}
\nabla_{E_2}^{k_2}\nabla_{E_3}^{k_3}A\psi(0)
\rangle k_1(k_1-1)\hspace{6.8 em} (VI)\\
&{}+\lambda\Re\langle\gamma,\nabla_{E_1}^{k_1}
\nabla_{E_2}^{k_2}\nabla_{E_3}^{k_3}A\psi(0)
\rangle k_2(k_2-1)\hspace{6.8 em} (VII)\\
&{}+\lambda\Re\langle\gamma,\nabla_{E_1}^{k_1}
\nabla_{E_2}^{k_2}\nabla_{E_3}^{k_3}A\psi(0)
\rangle k_3(k_3-1)\hspace{6.8 em} (VIII).
\end{align*}
Therefore we obtain the analogue of (\ref{taylor_dim2_2}) namely
\begin{displaymath}
  \Re\langle \gamma,\nabla_{E_{j_1}}...\nabla_{E_{j_r}}A\psi(0)\rangle=0
\end{displaymath}
for all $j_1,...,j_r\in\{1,2,3\}$. 
Thus there exist $a_{j_1,...,j_r,k}\in\mR$ such that 
\begin{displaymath}
  \nabla_{E_{j_1}}...\nabla_{E_{j_r}}A\psi(0)=\sum_{k=1}^3 a_{j_1,...,j_r,k}E_k\cdot\gamma.
\end{displaymath}
Observe that the coefficients $a_{j_1,...,j_r,k}$ are symmetric in the first $r$ indices. 
We insert this into (\ref{taylor_dim3_1}) and we obtain
\begin{equation}
  \label{taylor_dim3_4}
  0=a_{j_1,...,j_r,i}+\sum_{k=1}^r a_{i,j_1,...,\widehat{j_k},...,j_r,j_k}
\end{equation}
for all $j_1,...,j_r,i\in\{1,2,3\}$. 
On the other hand since $\psi\in\ker(D^g-\lambda)$ 
we find using the induction hypothesis
\begin{align}
\nonumber
0&=\lambda\nabla_{E_{j_1}}...\nabla_{E_{j_{r-1}}}
A\psi(0)\\
\nonumber
&=\nabla_{E_{j_1}}...\nabla_{E_{j_{r-1}}}
\sum_{i=1}^3 E_i\cdot\nabla_{E_i}A\psi(0)\\
\nonumber
&=\sum_{i,k=1}^3 a_{j_1,...,j_{r-1},i,k}E_i
\cdot E_k\cdot\gamma\\
\label{taylor_dim3_5}
&=-\sum_{i=1}^3 a_{j_1,...,j_{r-1},i,i}\gamma
+\sum_{i,k=1\atop_{i<k}}^3
(a_{j_1,...,j_{r-1},i,k}-a_{j_1,...,j_{r-1},k,i})
E_i\cdot E_k\cdot\gamma
\end{align}
for $j_1,...,j_{r-1}\in\{1,2,3\}$. 
Consider $a_{j_1,...,j_r,i}$ with 
$j_1,...,j_r,i\in\{1,2,3\}$. 
If~$j_k=i$ for all $k\in\{1,...,r\}$ then 
by (\ref{taylor_dim3_4}) we know that 
$a_{j_1,...,j_r,i}=0$. 
If there exists $k$ such that $j_k\neq i$ it follows 
from the coefficient of $E_{j_k}\cdot E_i\cdot\gamma$ 
in~(\ref{taylor_dim3_5}) that
\begin{displaymath}
a_{i,j_1,...\widehat{j_k}...,j_r,j_k}
=a_{j_1,...,j_r,i}.
\end{displaymath}
Again (\ref{taylor_dim3_4}) yields 
$a_{j_1,...,j_r,i}=0$. 
We conclude that all $a_{j_1,...,j_r,i}$ vanish and 
that~$\nabla^r A\psi(0)=0$. 
This proves the assertion in the case $n=3$.
\end{proof}

\begin{remark}
\label{remark_higher_dim}
It is not clear how to prove this lemma 
for~$n\geq4$. 
Namely the condition 
$\Re\langle\gamma,\nabla_{E_i}A\psi(0)\rangle=0$ 
for all~$i$ leads to
\begin{displaymath}
  \nabla_{E_i}A\psi(0)=\sum_{k=1}^n 
  a_{ik}E_k\cdot\gamma+\sum_{k=1}^n 
  b_{ik}\cdot\gamma
\end{displaymath}
with $a_{ik}\in\mR$ and 
with elements $b_{ik}$ of the Clifford algebra. 
As in (\ref{taylor_dim2_1}),  
(\ref{taylor_dim3_1}) with $r=1$ 
it follows that $a_{ik}=-a_{ki}$ 
for all $i$, $k$ and furthermore
\begin{align*}
0&=\lambda A\psi(0)\\
&=\sum_{i=1}^n E_i\cdot\nabla_{E_i}A\psi(0)\\
&=2\sum_{i,k=1\atop{i<k}}^n a_{ik}E_i
\cdot E_k\cdot\gamma+\sum_{i,k=1}^n E_i
\cdot b_{ik}\cdot\gamma.
\end{align*}
But for $n\geq4$ the spinors $E_1\cdot E_2\cdot\gamma$ 
and $E_3\cdot E_4\cdot\gamma$ 
are not linearly independent in general. 
Thus we cannot conclude immediately that all 
the~$a_{ik}$ vanish.
\end{remark}

\begin{proof}[Proof of Theorem 
\ref{theorem_no_zeros}]
We prove that the subspace 
$S_0(M,\Theta)\cap N_m(M)$ is dense in~$R(M)$. 
Let~$U\subset R(M)$ be open. 
It is sufficient to show that 
$U\cap S_0(M,\Theta)\cap N_m(M)$ is not empty. 
Since by Dahl's result the subspace 
$S_0(M,\Theta)\cap S_m(M)$ is dense in~$R(M)$, 
there exists a metric~$g$ in 
$U\cap S_0(M,\Theta)\cap S_m(M)$. 
Let $V\subset R(M)$ be an open neighborhood 
of~$g$ as in Lemma~\ref{lemma_F_psi}. 
We may assume that $V\subset U\cap S_0(M,\Theta)$. 
Let~$\lambda$ be one of the 
eigenvalues $\{\lambda_{-m}(g),...,\lambda_m(g)\}$ 
of~$D^g$ 
and let~$\psi$ be an~$L^2$-normalized eigenspinor 
corresponding to~$\lambda$. 
Define 
\begin{displaymath}
  F_{\psi}:\quad V\to C^{\infty}(\Sigma^gM)
\end{displaymath}
as in Lemma \ref{lemma_F_psi}.
We show that a suitable restriction 
of~$F_{\psi}^{ev}$ is 
transverse to the zero section of~$\Sigma^gM$. 
Let~$p\in M$ with~$\psi(p)=0$. 
By Lemma~\ref{lemma_transverse} 
there exist $f_{p,1},...,f_{p,4}\in C^{\infty}(M,\mR)$ 
such that we have 
\begin{displaymath}
\pi_1(dF_{\psi}^{ev}|_{(g,p)}
(T_gV_{f_{p,1}...f_{p,4}}\oplus\{0\}))=\Sigma^g_p M.
\end{displaymath}
By continuity of~$dF_{\psi}^{ev}$ there 
exists an open neighborhood $U_p\subset M$ of~$p$ 
such that for all~$q\in U_p$ with~$\psi(q)=0$ 
we have
\begin{displaymath}
\pi_1(dF_{\psi}^{ev}|_{(g,q)}
(T_gV_{f_{p,1}...f_{p,4}}\oplus\{0\}))=\Sigma^g_q M.
\end{displaymath}
For all $p\in M$ with $\psi(p)=0$ we choose an 
open 
neighborhood~$U_p\subset M$ as above. 
Since the zero set of~$\psi$ is compact, there exist 
finitely many points $p_1$,...,$p_r\in M$ 
and there exist 
open neighborhoods~$U_{p_i}\subset M$ 
of~$p_i$ and $f_{p_i,1},...,f_{p_i,4}\in 
C^{\infty}(M,\mR)$, 
$1\leq i\leq r$, 
such that for every~$i$ we have~$\psi(p_i)=0$ and 
such that the open neighborhoods 
$U_{p_1}$,...,$U_{p_r}$ 
cover the zero set 
of~$\psi$ and such that for every $i\in\{1,...,r\}$ 
and 
for every~$q\in U_{p_i}$ with~$\psi(q)=0$ we have 
\begin{displaymath}
\pi_1(dF_{\psi}^{ev}|_{(g,q)}
(T_gV_{f_{p_i,1}...f_{p_i,4}}\oplus\{0\}))
=\Sigma^g_q M.
\end{displaymath}
We label the functions $f_{p_i,j}$ by 
$f_1,...,f_{4r}$. 
For every one of the finitely many eigenvalues 
$\{\lambda_{-m}(g),...,\lambda_m(g)\}$ of $D^g$ 
we choose an eigenspinor and we repeat this 
procedure. 
Then for every $i\in\{-m,...,m\}\setminus\{0\}$ 
we obtain functions 
$f_{i,1},...,f_{i,4r_i}$ 
as above, where~$r_i\in\mN$ might depend on~$i$. 
We label these functions by 
\begin{displaymath}
\{f_1,...,f_s\}:=
\bigcup_{i=-m\atop i\neq0}^{m}
\{f_{i,1},...,f_{i,4r_i}\}.
\end{displaymath}
We define 
\begin{displaymath}
  V_{f_1...f_s}:=\Big\{\Big(1+\sum_{i=1}^{s} t_i
  f_i\Big)g\Big|\,t_1,...,t_s\in\mR\Big\}\cap V.
\end{displaymath}
and we define $F_{\psi_i}$: 
$V_{f_1...f_s}\to C^{\infty}(\Sigma^gM)$, 
$i\in\{-m,...,m\}\setminus\{0\}$, 
as in Lemma~\ref{lemma_F_psi}. 
Then for every $i\in\{-m,...,m\}\setminus\{0\}$ 
and for every~$q\in M$ with~$\psi_i(q)=0$ we have 
\begin{displaymath}
\pi_1(dF_{\psi_i}^{ev}|_{(g,q)}
(T_gV_{f_1...f_s}\oplus\{0\}))=\Sigma^g_q M. 
\end{displaymath}
By continuity of the maps $dF_{\psi_i}^{ev}$ 
there exists an open 
neighborhood 
$W\subset V_{f_1...f_s}$ of~$g$ such that 
for every $i\in\{-m,...,m\}\setminus\{0\}$ 
and for every $(h,q)\in W\times M$ 
with $F_{\psi_i}^{ev}(h,q)=0$ we have 
\begin{displaymath}
\pi_1(dF_{\psi_i}^{ev}|_{(h,q)}
(T_hV_{f_1...f_s}\oplus\{0\}))=\Sigma^g_q M.
\end{displaymath}
It follows that the restrictions of all the 
maps~$F_{\psi_i}^{ev}$ to $W\times M$ 
are transverse to the zero section of~$\Sigma^gM$. 

For every $i\in\{-m,...,m\}\setminus\{0\}$ 
we define~$Y_i$ as the subset of all $h\in W$ 
such that~$F_{\psi_i}(h)$ is nowhere zero on~$M$. 
By Remark~\ref{remark_transverse} the subset~$Y_i$ 
is the subset of all $h\in W$ such 
that~$F_{\psi_i}(h)$ 
is transverse to the zero section. 
By Lemma~\ref{lemma_F_psi} the restrictions 
of the maps $F_{\psi_i}^{ev}$ to $W\times M$ are 
in $C^1(W\times M,\Sigma^gM)$ and thus the 
condition~$r>0$ in Theorem~\ref{param_transvers} 
is satisfied. 
By this theorem all the subsets~$Y_i$ 
are dense in~$W$. 
Since the zero section is closed in~$\Sigma^gM$ and 
all the 
maps~$F_{\psi_i}^{ev}$ are continuous, all the 
subsets~$Y_i$ are open in~$W$. 
Thus the intersection 
\begin{displaymath}
Y:=\bigcap_{i=-m\atop i\neq0}^{m}Y_i
\end{displaymath}
is open and dense in~$W$. 
By definition we have $Y=N_m(M)\cap W$. 
Since we have $W\subset U$ the intersection 
$S_0(M,\Theta)\cap N_m(M)\cap U$ 
is not empty. 
Thus the subset 
$S_0(M,\Theta)\cap N_m(M)$ is dense in~$R(M)$. 
The proof that for every~$g$ in~$R(M)$ 
the subset $N_m(M)\cap[g]$ is dense in~$[g]$ 
is analogous if we use that $S_m(M)\cap[g]$ 
is dense in~$[g]$ by Dahl's result.
\end{proof}

\subsection{Harmonic spinors on closed surfaces}
\label{surfaces_section}

In this section we give a counterexample showing 
that Theorem \ref{main_theorem} 
does not hold for harmonic spinors in the case 
$n=2$. 
Let $(M,g,\Theta)$ be a closed Riemannian spin 
manifold of dimension~$2$. 
The spinor bundle splits as 
\begin{displaymath}
  \Sigma^gM=\Sigma^+M\oplus\Sigma^-M
\end{displaymath}
and sections of $\Sigma^\pm M$ will be called 
positive respectively negative spin\-ors. 
The manifold~$(M,g)$ is K\"ahler and the 
bundle~$\Sigma^+M$ is canonically isomorphic 
to a holomorphic line bundle~$L$ on~$M$. 
Furthermore positive harmonic spinors can be 
identified with holomorphic sections 
of~$L$ (see e.\,g.\,\cite{hitchin:74}, 
\cite{baer.schmutz:92}). 

To every positive or negative spinor on $(M,g)$ one 
can assign a vector field 
on~$M$ by a method given in \cite{ammann:98} which we 
briefly recall. 
First we define the maps 
$\tau_{\pm}$: $\SO(2)\to\mC$ by
\begin{displaymath}
\begin{pmatrix}
 \cos t & -\sin t \\ \sin t & \cos t
\end{pmatrix}
\mapsto\exp(\pm it).
\end{displaymath}
We define a complex structure~$J$ on~$M$ such that 
for every~$p\in M$ and for every unit vector  
$X\in T_pM$ the system $(X,JX)$ is a positively 
oriented orthonormal basis of~$T_pM$. 
Then the maps 
\begin{displaymath}
\rP_{\SO}(M,g)\times_{\tau\pm}\mC\to(TM,\mp J),\quad 
[(e_1,e_2),1]\mapsto e_1
\end{displaymath}
are isomorphisms of complex line bundles and the 
following holds.

\begin{lemma}[\cite{ammann:98}]
\label{spinor_to_vector}
Let $(M,g,\Theta)$ be a Riemannian spin manifold 
of dimension $2$. Then the maps 
\begin{align*}
\Phi_{\pm}:\quad\Sigma^{\pm}M
=\rP_{\Spin}(M,g)\times_{\rho}\Sigma^{\pm}_2
&\to
\rP_{\SO}(M,g)\times_{\tau_{\mp}}\mC\cong(TM,\mp J)\\
{}[s,\sigma]&\mapsto[\Theta(s),\sigma^2]
\end{align*}
are well defined.
\end{lemma}

We denote by~$\gamma$ the genus of~$M$. 
Assume that~$\psi$ is a positive harmonic 
spinor on~$M$ and that 
$p\in M$ is a point with $\psi(p)=0$. 
After a choice of a local holomorphic chart of~$M$ 
and of a local 
trivialization of the holomorphic line 
bundle~$\Sigma^+M$ around~$p$ 
the spinor~$\psi$ corresponds locally to a 
holomorphic function. 
We define~$m_p$ as the order of the zero~$p$. 
Let~$X$ be the vector field on~$M$ assigned 
to~$\psi$ via Lemma \ref{spinor_to_vector}. It 
follows that~$X$ 
has an isolated zero at~$p$ with index equal 
to~$-2m_p$. 
Let $\chi(M)=2-2\gamma$ denote the Euler 
characteristic of~$M$. 
Denote by~$N$ the zero set of~$\psi$. 
Since~$M$ is compact, 
the set~$N$ is finite. By the Poincar\'e-Hopf 
Theorem we obtain the following result.

\begin{theorem}
\label{theorem_surfaces}
Assume that $\psi$ is a positive harmonic spinor on 
a 
closed surface $(M,g,\Theta)$ and let $N\subset M$ 
be its zero set. Then $N$ is finite and we have
\begin{displaymath}
  \sum_{p\in N}m_p=-\frac{1}{2}\chi(M)=\gamma-1.
\end{displaymath}
\end{theorem}

On every closed oriented surface $M$ 
of genus~$2$ there exists a spin structure, 
such that for every Riemannian 
metric~$g$ on~$M$ we have $\dim_{\mC}\ker(D^g)=2$ 
(see Proposition 2.3 in \cite{hitchin:74} and its 
proof). 
We take such a spin structure. 
Then by Theorem \ref{theorem_surfaces} for every 
Riemannian metric~$g$ on~$M$ every 
positive harmonic spinor of~$D^g$ vanishes 
at exactly one point. 
Thus Theorem \ref{main_theorem} does not hold for 
harmonic spinors in the case~$n=2$.

\subsection{Harmonic spinors in dimensions $4k$, $k\geq 1$}

In this section we give examples showing that 
Theorem \ref{main_theorem} does 
not hold for harmonic spinors in the 
case~$n=4k$, $k\geq 1$. 

Let~$M$ be a closed spin manifold of dimension~$4k$, 
$k\geq 1$. 
It follows from the Atiyah-Singer index theorem 
that for every Riemannian metric $g$ on $M$ we have 
\begin{equation}
\label{A_hat_lower_bound}
\dim_{\mC}\ker(D^g)\geq|\hat{A}(M)|,
\end{equation}
where~$\hat{A}(M)$ is an integer-valued invariant 
which depends on the spin-bordism class of~$M$ but 
not on the metric 
(see e.\,g.\,\cite{lawson.michelsohn:89}, 
Thm.\,III.13.10, p.\,256). 
Now for every integer~$k\geq 1$ and for every 
even integer~$d\geq 2$ let~$p$ be 
a polynomial function on~$\mC^{2k+2}$ with 
complex coefficients which is homogeneous of 
degree~$d$ and such that for every 
$z\in\mC^{2k+2}\setminus\{0\}$ with $p(z)=0$ 
we have $\nabla p(z)\neq 0$. 
We define 
\begin{displaymath}
V^{2k}(d):=\{[z_0:...:z_{2k+1}]\in\mC P^{2k+1}|\,
p(z_0,...,z_{2k+1})=0\}.
\end{displaymath}
As explained on pages 88 and 138 of 
\cite{lawson.michelsohn:89} 
the space~$V^{2k}(d)$ is a closed spin manifold of 
real dimension~$4k$ and we have 
\begin{displaymath}
\hat{A}(V^{2k}(d))=\frac{2^{-2k}d}{(2k+1)!}
\prod_{j=1}^{k}(d^2-4j^2).
\end{displaymath}
Thus we can choose~$d$ large enough such that for 
every Riemannian metric~$g$ on~$V^{2k}(d)$ the 
dimension of $\ker D^g$ is greater than 
the rank of the spinor bundle 
by (\ref{A_hat_lower_bound}). 
In particular for every Riemannian 
metric~$g$ on~$V^{2k}(d)$ 
there exists a harmonic spinor with non-empty 
zero set. 
We remark that~$V^{2k}(d)$ is connected since 
$k\geq 1$. 
This shows that Theorem~\ref{main_theorem} does 
not hold for harmonic spinors in dimensions~$4k$, 
$k\geq 1$. 


\end{document}